\documentclass{amsart}
\usepackage{graphicx}
\usepackage{amssymb}
\usepackage{url}
\usepackage{bbm}

\renewcommand{\i}{\mathbf{i}}

\newcommand{\E}{\mathbb{E}}

\newcommand{\C}{\mathbb{C}}
\newcommand{\R}{\mathbb{R}}

\theoremstyle{plain}
\newtheorem{lemma}{Lemma}[section]
\newtheorem{theorem}[lemma]{Theorem}
\newtheorem{corollary}[lemma]{Corollary}
\newtheorem{proposition}[lemma]{Proposition}

\theoremstyle{remark}
\newtheorem{remark}[lemma]{Remark}

\title{The rectangular finite free heat flow}

\author{Cesar Cuenca}
\address{Department of Mathematics, The Ohio State University, Columbus, OH 43210, USA}
\email{cesar.a.cuenk@gmail.com}

\author{Colin McSwiggen}
\address{Institute of Mathematics, Academia Sinica, Da'an District, Taipei 106319, Taiwan}
\email{csm@as.edu.tw}

\begin{document}

\begin{abstract}
    We define and study the rectangular finite free heat flow, a dynamical system on polynomials that plays the role of the heat equation in the setting of rectangular finite free probability.  We show several equivalent characterizations of the evolution (including PDE and gradient flow formulations), establish basic properties of the dynamics, and determine the asymptotic distributions of the polynomial roots in the long-time and high-degree limits. We also discuss connections with Calogero--Moser systems and Dunkl processes, and we show that the rectangular finite free heat flow describes the mean curvature expansion of a family of compact Lie group orbits.
\end{abstract}

\maketitle

\tableofcontents

\pagebreak

\section{Introduction}

The heat equation on the real line is the partial differential equation
\begin{equation} \label{eqn:heat}
    \partial_t u(x,t) = \frac{1}{2} \partial^2_x u(x,t), \qquad x \in \R, \quad t > 0.
\end{equation}
It plays a fundamental role in probability theory due to its close connection with Brownian motion: if $u(x,t)$ solves (\ref{eqn:heat}) with initial condition $u(x,0) = f(x)$ satisfying some mild growth and regularity conditions, then
\begin{equation}
    u(x,t) = \E[f(x + B_t)]
\end{equation}
for all $x\in\R$, $t>0$, where $(B_t)_{t\ge 0}$ is a standard Brownian motion.
The Gaussian distribution of variance $t$, with density
\[
\varphi_t(x) = \frac{1}{\sqrt{2\pi t}} e^{-x^2/2t},
\]
is the law of $B_t$ for any fixed $t > 0$.  It is also the fundamental solution of the heat equation, in the sense that the solution $u(x,t)$ with initial condition $f(x)$ can be expressed as
\[
u(x,t) = (f * \varphi_t)(x) = \int_{\R} f(y) \varphi_t(x-y) \, \mathrm{d}y
\]
for all $t > 0$.  We can thus regard the \emph{heat flow} on $\R$ as a dynamical system on distributions: given an initial distribution $\mu$, its state at time $t$ is $\mu * \varphi_t$, the convolution of $\mu$ with a Gaussian distribution of variance $t$.

Recent decades have seen the development of new areas of probability theory in which the roles of the Gaussian distribution and the convolution of measures are played by other analogous objects.  In any such setting, the proper analogue of ``convolution with a Gaussian distribution of variance $t$'' gives rise to a dynamical system that can be regarded as a kind of heat flow. Perhaps the best known example comes from \emph{free probability}, a noncommutative generalization of classical probability theory originally developed by Voiculescu in the 1980s, which now plays a major role in random matrix theory and related fields \cite{voiculescu1992free}.  In free probability, the notion of independence of random variables is replaced by an alternative notion called \emph{freeness}. Similarly to independent random variables, there is a convolution operation that describes the addition of free random variables: if $X$ and $Y$ are two free random variables with distributions $\mu$ and $\nu$ on $\R$, then the distribution of $X + Y$ is called the \emph{additive free convolution} of $\mu$ and $\nu$, and is denoted by $\mu \boxplus \nu$. Normalized sums of free random variables with mean 0 and variance $t$ also satisfy a central limit theorem \cite{voiculescu1991limit}: their distributions converge to the \emph{semicircle distribution} of variance $t$,
\[
\sigma_t(\mathrm{d} x) = \mathbbm{1}_{[-2\sqrt{t}, 2\sqrt{t}]}(x)\, \frac{\sqrt{4t-x^2}}{2\pi t} \, \mathrm{d}x.
\]
In this sense, the semicircle distribution is the noncommutative analogue of the Gaussian distribution.  The corresponding version of the heat flow, called the \emph{free heat flow}, is the measure-valued dynamical system defined by
\begin{equation} \label{eqn:freehf}
    \mu_0 = \mu,\qquad \mu_t = \mu \boxplus \sigma_t, \quad t > 0,
\end{equation}
where $\mu$ can be any compactly supported probability measure on $\R$.

The system (\ref{eqn:freehf}) was first studied by Biane \cite{biane1997free}, who showed that the Cauchy transform
\begin{equation} \label{eqn:cauchy-transf}
G(z,t) = \int_{\R} \frac{1}{z-x} \,\mu_t(\mathrm{d}x), \qquad z \in \C, \quad \mathrm{Re}(z) > 0
\end{equation}
satisfies the complex Burgers equation:
\begin{equation} \label{eqn:cx-burgers}
\partial_t G + G \partial_z G = 0.
\end{equation}
Just as the classical heat flow can be formulated as a gradient flow for the Boltzmann entropy with respect to the $L^2$-Wasserstein metric \cite{JordanKinderlehrerOtto1998, Otto2001}, the free heat flow can be regarded, at least on a formal level, as a gradient flow for the \emph{free entropy}
\begin{equation} \label{eqn:free-ent}
\Sigma(\mu) = \iint_{\R} \log |x-y| \, \mu(\mathrm{d}x) \, \mu(\mathrm{d}y) 
\end{equation}
defined in \cite{voiculescu1994free}, with respect to a free analogue of the Wasserstein metric defined in \cite{BianeVoiculescu2001}.

The above constructions of classical and free probability all have analogues in \emph{finite} free probability, a finite-dimensional approximation to free probability theory in which probability measures are replaced by monic polynomials of degree $d$ \cite{arizmendi2018cumulants, edelman2005polynomial, marcus2021polynomial, MarcusSpielmanSrivastava2022}. Such a polynomial $p(x) = \prod_{i=1}^d (x-a_i)$ is uniquely determined by the empirical distribution of its roots, which is the finitely supported probability measure
\[
\frac{1}{d} \sum_{i=1}^d \delta_{a_i}.
\]
From this perspective, one can speak of approximating a probability measure by monic polynomials as the degree grows to infinity. Finite free probability defines operations for monic polynomials that approximate corresponding operations for measures defined in free probability.  Notably, the finite free convolution of two real-rooted polynomials $p(x) = \prod_{i=1}^d (x-r_i)$, $q(x) = \prod_{i=1}^d (x-s_i)$ of the same degree $d$ is defined as
\begin{equation} \label{eqn:ffconv}
(p \boxplus_d q)(x) = \E\big[ \det\big(xI - (A + UBU^*)\big) \big],
\end{equation}
the expected characteristic polynomial of $A + UBU^*$, where $A$ and $B$ are any fixed Hermitian matrices with respective spectra $(r_1,\hdots,r_d)$ and $(s_1, \hdots, s_d)$, and $U \in \mathrm{U}(d)$ is a random Haar $d\times d$ unitary matrix.

In finite free probability, the analogue of the semicircle distribution is the degree-$d$ Hermite polynomial of variance $t$,
\begin{equation}\label{eq:hermite_variance_t}
H_d^{[t]}(x) = (-t)^d e^{x^2/2t} \frac{d^d}{dx^d} e^{-x^2/2t},
\end{equation}
whose root distribution, properly scaled, approximates the semicircle distribution of variance $t$ as $d\to\infty$.  The \emph{finite free heat flow} with initial condition $p(x)$ is thus a dynamical system on monic polynomials of degree $d$, defined by
\begin{equation} \label{eqn:ffhf}
p(x,t) = \big(p \boxplus_d H_d^{[t]}\big)(x), \qquad t \ge 0.
\end{equation}
Remarkably, the polynomial $p(x,t)$ defined by (\ref{eqn:ffhf}) also solves the heat equation, but with time reversed:
\begin{equation} \label{eqn:polyheat}
    \partial_t p(x,t) = -\frac{1}{2} \partial_x^2 p(x,t).
\end{equation}
The finite free heat flow has received significant recent attention in free probability and random matrix theory  \cite{hall2025heat, hall2025heatgaf, hall2025zeros, hoefert2025zeros, kabluchko2025lee}, but the study of this system in the form (\ref{eqn:polyheat}) dates back nearly 50 years: Calogero studied the evolution of zeros of polynomials under the heat equation as early as 1978 \cite{calogero1978motion} and later gave a more comprehensive account of the evolution of polynomial zeros under first- and second- order linear differential operators \cite[Section 2.3]{calogero2001classical}, while independently, in the 1990s, Csordas, Smith and Varga related the zeros of polynomial solutions of the heat equation to Lehmer pairs of zeros of the Riemann zeta function \cite{csordas1994lehmer}. For further details on the root evolution of (\ref{eqn:polyheat}) and related topics, see the blog post by Tao \cite{tao2017heatflow} or the lecture notes by Sokal \cite{sokal2023motion}.

\smallskip

In this paper, we define and study the \emph{rectangular finite free heat flow}, which is the analogue of (\ref{eqn:ffhf}) in rectangular finite free probability.  Rectangular finite free probability is a generalization of finite free probability in which the random matrices considered are rectangular rather than square \cite{Cuenca2024, gribinski2024theory, gribinski2022rectangular}.  It provides finite-dimensional approximations to the constructions of rectangular free probability, which is itself a generalization of free probability theory developed for studying the asymptotic singular value distributions of rectangular random matrices \cite{benaych2009rectangular, BenaychGeorges2007, BenaychGeorges2009entropy, xu2023rectangular}. We review some basic notions from rectangular free probability in Section \ref{sec:high-degree} below.

In rectangular finite free probability, our main setting of interest, the finite free convolution (\ref{eqn:ffconv}) is replaced by the rectangular finite free convolution (henceforth simply ``rectangular convolution''), defined as follows. Let $p(x) = \prod_{i=1}^d (x-r_i^2)$, $q(x) = \prod_{i=1}^d (x-s_i^2)$ be monic polynomials of the same degree $d$ with nonnegative real roots, and fix an integer $n\ge d$.  The \emph{$(n,d)$-rectangular convolution} of $p$ and $q$ is defined as
\begin{equation} \label{eqn:rectffconv}
(p \boxplus_d^n q)(x) \;=\;  \E\big[ \det\big(xI - (A + UBV^*)(A + UBV^*)^*\big) \big],
\end{equation}
the expected characteristic polynomial of $(A + UBV^*)(A + UBV^*)^*$, where $A$ and $B$ are fixed $d \times n$ matrices with respective squared singular values $(r_1^2, \hdots, r_d^2)$ and $(s_1^2, \hdots, s_d^2)$, while $U \in \mathrm{U}(d)$, $V \in \mathrm{U}(n)$ are independent random Haar unitary matrices.

The role of the Gaussian distribution in rectangular finite free probability is played by the \emph{Laguerre heat polynomials} of parameter $t\ge 0$,\footnote{The Laguerre heat polynomials $G_{n,d}^{[t]}(x)$ are orthogonal with respect to the probability measure with density $\frac{1}{n!(2t)^{n+1}}\,x^ne^{-x/2t}$ (a rescaled Laguerre weight), of mean $2(n+1)t$. By contrast, the Hermite polynomials $H_d^{[t]}(x)$ in \eqref{eq:hermite_variance_t} are orthogonal with respect to the Gaussian distribution of mean zero and variance $t$.}
\begin{equation} \label{eqn:Gdef}
    G_{n,d}^{[t]}(x) = (-2t)^{d}\cdot d!\cdot L_d^{(n)}\left(\frac{x}{2t}\right),
\end{equation}
where
\begin{equation} \label{eqn:Lpoly}
    L_d^{(\alpha)}(x) = \sum_{k=0}^d{ \frac{(d+\alpha)^{\downarrow k}}{(d-k)!k!}(-x)^{d-k} }
\end{equation}
denotes the Laguerre polynomial of degree $d$ with parameter $\alpha > -1$, and $m^{\downarrow k}$ denotes the lowering factorial:
\[
    m^{\downarrow k} = \begin{cases}
        m(m-1)\cdots(m-k+1),&\text{ if }k=1,2,\dots,\\
        1,&\text{ if }k=0.
    \end{cases}
\]

The \emph{rectangular finite free heat flow} with initial condition $p(x) = \prod_{i=1}^d (x - r_i^2)$ is then defined by
\begin{equation} \label{eqn:BC-FFC}
p(x,t) = \big( p \boxplus^n_d G_{n,d}^{[t]} \big) (x), \qquad t \ge 0.
\end{equation}
The purpose of this paper is to establish a number of fundamental characterizations, properties, and interpretations of the dynamical system (\ref{eqn:BC-FFC}).  We show that this system has a rich theory that provides finite-dimensional models for many phenomenological features of diffusion processes, and that it also arises naturally in connection with several other areas of mathematics, from integrable systems to the differential geometry of homogeneous spaces.

\subsection{Overview and main results}

In Section \ref{sec:basic} we establish several equivalent characterizations of the rectangular finite free heat flow: an explicit formula for the coefficients, a PDE formulation showing that $p(x,t)$ solves a time-reversed Fokker--Planck equation for a squared Bessel process, and a system of ODEs for the evolution of the roots, which can be recast as a gradient flow for a rectangular analogue of the free entropy.

In Section \ref{sec:properties} we study basic properties of the dynamics: we show that the entropy and Fisher information increase and decrease in time, respectively; that the roots are distinct and positive for all $t > 0$; that the center of mass of the roots moves linearly to the right; and that the Cauchy transform of the root distribution satisfies a kind of generalized Burgers equation.

In Section \ref{sec:asymptotics} we study the long-time and high-degree asymptotics of the root distribution.  We prove an analogue of the fact that the Gaussian distribution is the unique fixed point of the classical heat flow after rescaling space by $\sqrt{t}$, and we show that the high-degree limit of the rectangular finite free heat flow is described by the rectangular free convolution with a Marchenko--Pastur distribution.

In Section \ref{sec:further} we consider alternative views of the rectangular finite free heat flow from the perspectives of integrable systems (as a first-order analogue of a Calogero--Moser system of type $BC$), stochastic processes related to special functions (as a freezing limit of a radial Dunkl process of type $BC$), and differential geometry (as a description of a group orbit evolving by outward mean curvature flow in an ambient space).

For readers familiar with the theory of root systems, we point out that the finite free convolution (\ref{eqn:ffconv}) is naturally associated with the root system of type $A$, while the rectangular convolution (\ref{eqn:rectffconv}) is its counterpart of type $BC$.  Although we avoid a detailed discussion of root systems in this paper, we hope that this investigation will be a step towards a more general theory of polynomial convolutions associated to root systems, and that it will help to illuminate the links between finite free probability and the rich network of root-system-based constructions in related fields, such as integrable systems and special functions.

\section{Equivalent characterizations} \label{sec:basic}

This section presents multiple equivalent descriptions of the polynomial evolution (\ref{eqn:BC-FFC}): an explicit formula for the coefficients of $p(x,t)$, a characterization of $p(x,t)$ as a polynomial solution of a PDE, and a system of ODEs satisfied by the roots of $p(x,t)$, which turns out to be a gradient flow for an entropy-like quantity.

\subsection{Coefficient evolution}

The following general formula for the coefficients of the rectangular convolution $p \boxplus^n_d q$ was given in \cite{gribinski2022rectangular}.  In fact, in \cite[Definition 2.1]{gribinski2022rectangular} the formula below is used to \emph{define} the rectangular convolution for arbitrary monic polynomials of degree $d$; in \cite[Theorem 1.5 and Theorem 2.3]{gribinski2022rectangular} it is shown that this definition preserves nonnegative real-rootedness and coincides with the definition (\ref{eqn:rectffconv}) in terms of random matrices in the case that $p$ and $q$ have nonnegative real roots.  In cases where no assumptions are made on the roots of $p$ or $q$, we can thus regard the following proposition as the definition of $p \boxplus^n_d q$.

\begin{proposition}
    For monic degree-$d$ polynomials
    \[
    p(x) = x^d + \sum_{i=1}^d a_{2i} x^{d-i}, \qquad q(x) = x^d + \sum_{i=1}^d b_{2i} x^{d-i},
    \]
    their $(n,d)$-rectangular finite free convolution is
    \begin{equation} \label{eqn:bc-ffc-coeffs}
    (p \boxplus^n_d q)(x) = x^d + \sum_{k=1}^d x^{d-k} \sum_{\substack{i,j\ge 0\\i+j=k}} \frac{(d-i)!(d-j)!}{d!(d-k)!} \frac{(n+d-i)!(n+d-j)!}{(n+d)!(n+d-k)!} a_{2i} b_{2j},
    \end{equation}
    where we set $a_0 = b_0 = 1$.
\end{proposition}

From (\ref{eqn:Gdef}) and (\ref{eqn:Lpoly}), we get the following expression for the coefficients of $G_{n,d}^{[t]}$:
\begin{equation} \label{eqn:Gcoeffs}
    G_{n,d}^{[t]}(x) = \sum_{k=0}^d{ \frac{(-2)^k d^{\downarrow k}(d+n)^{\downarrow k}}{k!} t^kx^{d-k} }.
\end{equation}
Combining (\ref{eqn:bc-ffc-coeffs}) and (\ref{eqn:Gcoeffs}), we then obtain the following explicit formula for the coefficients of a polynomial evolving under the rectangular finite free heat flow.

\begin{corollary}
For any monic degree-$d$ polynomial $p(x) = x^d + \sum_{i=1}^d a_{2i} x^{d-i}$, the coefficients of $p(x,t) = (p \boxplus^n_d G_{n,d}^{[t]})(x)$ are given explicitly by
\begin{equation} \label{eqn:pxt-coeffs}
    p(x,t) = x^{d} + \sum_{k=1}^d \sum_{j=0}^k \frac{(-2)^j}{j!} \frac{(d-k+j)!}{(d-k)!} \frac{(n+d-k+j)!}{(n+d-k)!} a_{2(k-j)} t^j x^{d-k},
\end{equation}
where we set $a_0=1$.
\end{corollary}

\begin{remark}
    The formula (\ref{eqn:bc-ffc-coeffs}) allows us to verify that the monomial $x^d = G_{n,d}^{[0]}(x)$ is an identity element for the rectangular convolution:
    \begin{equation}\label{eq:dirac}
    p \boxplus^n_d x^d = p(x).
    \end{equation}
    We can thus regard $x^d$ as analogous to a Dirac mass at 0, which is precisely its root distribution.  Likewise, from (\ref{eqn:Gcoeffs}), we can verify that the Laguerre heat polynomials have the semigroup property:
    \begin{equation} \label{eqn:Gsemigp}
        G_{n,d}^{[s]} \boxplus^n_d G_{n,d}^{[t]} = G_{n,d}^{[s+t]}, \qquad s,t \ge 0.
    \end{equation}
\end{remark}

\subsection{PDE formulation}

We now give a partial differential equation for the evolution of $p \boxplus^n_d G_{n,d}^{[t]}$, which sheds some additional light on the analogy with classical probability.

\begin{theorem}\label{thm:pde}
For any monic degree-$d$ polynomial $p(x)$, the polynomial $p(x,t)$ defined by (\ref{eqn:BC-FFC}) satisfies the partial differential equation
\begin{equation}\label{eqn:pde}
{\partial_t} p(x,t) = -2\left( x {\partial_x^2} + (n+1) {\partial_x} \right) p(x,t).
\end{equation}
with initial condition $p(x,0) = p(x)$.
\end{theorem}

\begin{proof}
    Since both sides of (\ref{eqn:pde}) are polynomials in $x$ and $t$, it suffices to check that the equality holds coefficient-wise. Let $C_{k,j}$ denote the coefficient of $t^j x^{d-k}$ in $p(x,t)$. Note that, from \eqref{eqn:pxt-coeffs}, it follows that $C_{k,j}=0$ whenever $k<j$.
    
    On the one hand, the coefficient of $t^j x^{d-k} $ in $\partial_t p(x,t)$ is
    \[
    A = (j+1)C_{k,j+1}.
    \]
    On the other hand, using the fact that
    \[
    (x\partial_x^2 + (n+1)\partial_x)x^m = m(m+n)x^{m-1},
    \]
    the coefficient of $t^j x^{d-k} $ in the right-hand side of (\ref{eqn:pde}) is
    \[
    B = -2(d-k+1)(n+d-k+1)C_{k-1,j}.
    \]
    We must show that $A=B$. If $k<j+1$, then $A=B=0$. Thus it only remains to consider the case when $k\ge j+1$.
    From (\ref{eqn:pxt-coeffs}), we have:
    \[ A = \frac{(-2)^{j+1}}{j!} \frac{(d-k+j+1)!}{(d-k)!} \frac{(n+d-k+j+1)!}{(n+d-k)!}\, a_{2(k-(j+1))} \]
    and
    \begin{multline*} 
    B = -2(d-k+1)(n+d-k+1)\\
    \times\frac{(-2)^j}{j!} \frac{(d-k+1+j)!}{(d-k+1)!} \frac{(n+d-k+1+j)!}{(n+d-k+1)!}\, a_{2((k-1)-j)}.
    \end{multline*}
    These two expressions are indeed equal, verifying that $A=B$ in all cases, and therefore \eqref{eqn:pde} holds.
    The fact that the initial condition is given by $p(x,0) = \big( p \boxplus^n_d G_{n,d}^{[0]} \big) (x) = p\boxplus^n_d x^d = p(x)$ is exactly equation \eqref{eq:dirac}.
\end{proof}

\begin{remark} \label{rem:bessel}
An \emph{$m$-dimensional squared Bessel process} is the norm squared of a standard Brownian motion in $m$ dimensions.  The generator for such a process is
    \[
    \mathcal{L} = 2x \partial_x^2 + m {\partial_x}.
    \]
Accordingly, the PDE (\ref{eqn:pde}) can be interpreted as a time-reversed Fokker--Planck equation for a $2(n+1)$-dimensional squared Bessel process.  By analogy, observe that the finite free heat flow (\ref{eqn:ffhf}) solves the time-reversed heat equation (\ref{eqn:polyheat}), and the heat equation is the Fokker--Planck equation for Brownian motion. Details on squared Bessel processes, including a derivation of the generator, can be found in \cite[Chapter IX]{RevuzYor1999}.
\end{remark}

The above PDE characterization of the rectangular finite free heat flow yields the following alternate formula in terms of a differential operator acting on the initial condition $p(x)$.
This formula was previously noted in the proof of \cite[Main Result III]{Cuenca2024}.

\begin{corollary}\label{lem:exponential_derivative}
    If $p(x)$ is any monic polynomial of degree $d$, then
    \begin{equation} \label{eqn:exp-deriv}
    \big( p \boxplus^n_d G_{n,d}^{[t]} \big)(x) = \exp\Big(-2t\cdot x^{-n}\partial_x x^{n+1}\partial_x \Big) p(x).
    \end{equation}
\end{corollary}

\begin{proof}
    The operator 
    \[
    \exp \Big(-2t\cdot x^{-n}\partial_x x^{n+1}\partial_x \Big) = \sum_{k=0}^\infty \frac{1}{k!} \Big(-2t\cdot x^{-n}\partial_x x^{n+1}\partial_x \Big)^k
    \]
    is defined formally as a power series that terminates after $d$ terms when applied to the polynomial $p(x)$. Differentiating in $t$, we find
    \begin{multline*}
    \partial_t \exp\Big(-2t\cdot x^{-n}\partial_x x^{n+1}\partial_x \Big) p(x)\\
    = -2 x^{-n}\partial_x x^{n+1}\partial_x  \exp\Big(-2t\cdot x^{-n}\partial_x x^{n+1}\partial_x \Big) p(x) \\
    = -2 (x \partial_x^2 + (n+1) \partial_x) \exp\Big(-2t\cdot x^{-n}\partial_x x^{n+1}\partial_x \Big) p(x),
    \end{multline*}
    verifying that the right-hand side of (\ref{eqn:exp-deriv}) solves the PDE (\ref{eqn:pde}) with initial condition $p(x)$.  Since the operator $x \partial_x^2 + (n+1) \partial_x$ is degree reducing, polynomial solutions to (\ref{eqn:pde}) with prescribed initial data are unique, so we have proved the claim.
\end{proof}

\subsection{Root evolution and gradient flow formulation}

We now study how the roots of the polynomial $p \boxplus^n_d G_{n,d}^{[t]}$ evolve in time.  In this subsection, we assume that the roots are positive and distinct at some fixed time $t$.  However, as we show below in Theorem \ref{thm:root-separation}, it turns out that if the roots are all nonnegative at $t = 0$ then they shatter instantaneously: they are positive and distinct for all $t > 0$.

\begin{theorem} \label{thm:roots-ode}
If $p(x,t)$ is a polynomial of degree $d$ with distinct positive roots $r_1(t)^2 > \dots > r_d(t)^2 > 0$, and which satisfies equation \eqref{eqn:pde}, then
\begin{equation}\label{eq:ode}
\dot{r}_k(t) = \sum_{j\colon j\ne k}{ \left(\frac{1}{r_k(t)-r_j(t)} + \frac{1}{r_k(t)+r_j(t)}\right) } + \frac{n+1}{r_k(t)},
\end{equation}
for all $k=1,\dots,d$.
\end{theorem}
\begin{proof}
We can assume that $p(x,t)$ is monic, so that
\[
p(x,t) = \prod_{i=1}^d{ (x-r_i(t)^2) }.
\]
Then
\begin{align}
\frac{\partial_tp}{p} &= -2\sum_{i=1}^d{ \frac{r_i(t)\dot{r}_i(t)}{x-r_i(t)^2}}, \label{eq:partial1}\\
\frac{\partial_xp}{p} &= \sum_{i=1}^d{ \frac{1}{x-r_i(t)^2} }\label{eq:partial2},\\
\frac{\partial_x^2p}{p} &= 2\sum_{1\le i<j\le d}{ \frac{1}{(x-r_i(t)^2)(x-r_j(t)^2)} }.\label{eq:partial3}
\end{align}

The assumption~\eqref{eqn:pde} implies that $(x\partial_x^2+\frac{1}{2}\partial_t+(n+1)\partial_x)p=0$.
As a result, from \eqref{eq:partial1}, \eqref{eq:partial2} and \eqref{eq:partial3}, we deduce the identity
\[
\sum_{1\le i<j\le d}{ \frac{2x}{(x-r_i(t)^2)(x-r_j(t)^2)} } + \sum_{i=1}^d{ \frac{n+1-r_i(t)\dot{r}_i(t)}{x-r_i(t)^2} }=0.
\]
Let $1\le k\le d$ be arbitrary.
By inspecting the residue of this identity at $x=r_k(t)^2$, we get
\[
\sum_{j\colon j\ne k}{\frac{2r_k(t)}{r_k(t)^2-r_j(t)^2}}+(n+1-r_k(t)\dot{r}_k(t)) = 0.
\]
Moving $\dot{r}_k(t)$ to one side, we obtain exactly the equation \eqref{eq:ode}. This finishes the proof.
\end{proof}

The ODEs for the root evolution in fact describe a gradient flow for an entropy-like quantity.

\begin{proposition} \label{prop:ode-grad}
    The equations (\ref{eq:ode}) are equivalent to the vector equality
    \begin{equation} \label{eq:ode-grad}
    \dot{r} = \nabla H(r),
    \end{equation}
    where $\nabla$ indicates the gradient with respect to $(r_1, \dots, r_d)$, and
    \begin{equation} \label{eqn:rent-def}
    H(r) = \log \left[ \left( \prod_{i=1}^d r_i^{n+1} \right) \prod_{1 \le i<j \le d} (r_i - r_j)(r_i + r_j) \right]. 
    \end{equation}
\end{proposition}

\begin{proof}
    Direct calculation from (\ref{eq:ode}).
\end{proof}

    From the gradient flow formulation (\ref{eq:ode-grad}), it is clear that in our setting the functional of the roots $H(r(t))$ plays the same role as the Boltzmann entropy in classical probability or the free entropy (\ref{eqn:free-ent}) in free probability. We will thus refer to $H$ as the \emph{rectangular entropy}.

\section{Properties of the dynamics}
\label{sec:properties}

Our analysis of the properties of the root dynamics begins by studying the time evolution of the rectangular entropy.

\subsection{Entropy and Fisher information}

Taking the time derivative of $H(r(t))$ and substituting (\ref{eq:ode-grad}), we find that the rectangular entropy evolves according to
\[
\partial_t H(r(t)) = \nabla H(r(t)) \cdot \dot{r} = \| \nabla H(r(t)) \|^2.
\]
The above expression is positive except at critical points of $H$.  From the explicit expression for $\nabla H(r)$ on the right-hand side of (\ref{eq:ode}), it is easy to see that $H$ has no critical points in the region $r_1 > \dots > r_d > 0$. Thus we have shown:

\begin{corollary} \label{cor:dtH}
    Suppose that $r_1(t) > \dots > r_d(t) > 0$ satisfy the ODE system (\ref{eq:ode}) at some time $t > 0$. Then the rectangular entropy $H(r(t))$ is strictly increasing at time $t$.
\end{corollary}

The time derivative of the Boltzmann entropy of a probability measure evolving under the classical heat flow is given by the Fisher information. Accordingly, for the rectangular finite free heat flow, the analogue of the Fisher information is
    \begin{equation} \label{eqn:rect-fisher}
        I(r) = \| \nabla H(r) \|^2 = \sum_{k=1}^d \left( \sum_{j\colon j\ne k} \left( \frac{1}{r_k-r_j} + \frac{1}{r_k+r_j} \right) + \frac{n+1}{r_k} \right)^2,
    \end{equation}
the \emph{rectangular Fisher information}. We now show a simplified formula for $I(r)$.

\begin{proposition} \label{prop:rf-simple}
    For $r_1 > \dots > r_d > 0$, we have
    \begin{equation} \label{eqn:rf-simple}
        I(r) = \sum_{k=1}^d \frac{(n+1)^2}{r_k^2} + \sum_{1 \le i < j \le d} \left( \frac{2}{(r_i - r_j)^2} + \frac{2}{(r_i + r_j)^2} \right).
    \end{equation}
\end{proposition}

\begin{proof}
    For any $1\le k\le d$, we have
    \begin{multline}\label{squaring}
    \left( \sum_{j\colon j\ne k} \left( \frac{1}{r_k-r_j} + \frac{1}{r_k+r_j} \right) + \frac{n+1}{r_k} \right)^2
    = \frac{(n+1)^2}{r_k^2}\\
    +\sum_{j\colon j\ne k}\left( \frac{1}{(r_k-r_j)^2} + \frac{1}{(r_k+r_j)^2} \right) + 2(n+2)\sum_{j\colon j\ne k}{ \frac{1}{r_k^2-r_j^2} }.
    \end{multline}
    Note that
    \begin{equation}\label{zero-sum}
    \sum_{k=1}^d\ \sum_{j\colon j\ne k}{ \frac{1}{r_k^2-r_j^2}} = 0.
    \end{equation}
    From \eqref{eqn:rect-fisher} and \eqref{zero-sum}, we see that summing \eqref{squaring} over $k$ ranging from $1$ to $d$ results in
    \[
    I(r) = \sum_{k=1}^d \left( \frac{(n+1)^2}{r_k^2} + \sum_{j\colon j\ne k}\left( \frac{1}{(r_k-r_j)^2} + \frac{1}{(r_k+r_j)^2} \right) \right).
    \]
    The right-hand side above is readily seen to be equal to the right-hand side of \eqref{eqn:rf-simple}, concluding the proof.
\end{proof}

Although the rectangular entropy increases in time, the next result shows that it decelerates --- that is, the rectangular Fisher information is decreasing.

\begin{proposition}
    For $r_1(t) > \dots > r_d(t)>0$ satisfying (\ref{eq:ode}), we have
    \begin{equation*}
    \partial_t I(r(t)) = -2 \left[ \sum_{k=1}^d \frac{n+1}{r_k^2} \dot{r}_k^2 + \sum_{1\le i < j\le d} \left( \frac{(\dot{r}_i - \dot{r}_j)^2}{(r_i - r_j)^2} + \frac{(\dot{r}_i + \dot{r}_j)^2}{(r_i + r_j)^2} \right) \right] \le 0.
\end{equation*}
\end{proposition}

\begin{proof}
By the chain rule and the gradient flow equation (\ref{eq:ode-grad}), we obtain
\[
\partial_t I(r(t)) = \sum_{k=1}^d \frac{\partial I}{\partial r_k} \dot{r}_k = \sum_{k=1}^d \frac{\partial I}{\partial r_k} \frac{\partial H}{\partial r_k}.
\]
Since $I = \|\nabla H\|^2$, we have 
\[
\frac{\partial I}{\partial r_k} = 2\sum_{j=1}^d \frac{\partial^2 H}{\partial r_k \partial r_j} \frac{\partial H}{\partial r_j},
\]
giving
\[
\partial_t I(r(t)) = 2\sum_{1\le k,j\le d} \frac{\partial^2 H}{\partial r_k \partial r_j} \frac{\partial H}{\partial r_k} \frac{\partial H}{\partial r_j} = 2 \dot{r} \cdot (D^2H) \dot{r},
\]
where $D^2H$ is the Hessian of $H$. From (\ref{eqn:rent-def}), we compute the entries of the Hessian:
\begin{align*}
\frac{\partial^2 H}{\partial r_k^2} &= -\frac{n+1}{r_k^2} - \sum_{j\colon j\ne k}\left[\frac{1}{(r_k-r_j)^2} + \frac{1}{(r_k+r_j)^2}\right], \\
\frac{\partial^2 H}{\partial r_k \partial r_j} &= \frac{1}{(r_k-r_j)^2} - \frac{1}{(r_k+r_j)^2}, \quad k \neq j.
\end{align*}
Expanding the dot product, we find:
\begin{align*}
\partial_t I(r(t)) &= 2 \dot{r} \cdot (D^2H) \dot{r} = 2\sum_{k=1}^d \frac{\partial^2 H}{\partial r_k^2} \dot{r}_k^2 + 2\sum_{k \neq j} \frac{\partial^2 H}{\partial r_k \partial r_j} \dot{r}_k \dot{r}_j \\
&= -2\sum_{k=1}^d \frac{n+1}{r_k^2} \dot{r}_k^2 - 2\sum_{k \neq j}\left[\frac{1}{(r_k-r_j)^2} + \frac{1}{(r_k+r_j)^2}\right] \dot{r}_k^2 \\
&\quad + 2\sum_{k \neq j} \left[\frac{1}{(r_k-r_j)^2} - \frac{1}{(r_k+r_j)^2}\right] \dot{r}_k \dot{r}_j.
\end{align*}
For the double sums, note that
\begin{align*}
\sum_{k \neq j} \frac{\dot{r}_k^2}{(r_k-r_j)^2} - \sum_{k \neq j} \frac{\dot{r}_k \dot{r}_j}{(r_k-r_j)^2} &= \sum_{1\le i < j\le d} \frac{(\dot{r}_i - \dot{r}_j)^2}{(r_i-r_j)^2},\\
\sum_{k \neq j} \frac{\dot{r}_k^2}{(r_k+r_j)^2} + \sum_{k \neq j} \frac{\dot{r}_k \dot{r}_j}{(r_k+r_j)^2} &= \sum_{1\le i < j\le d} \frac{(\dot{r}_i + \dot{r}_j)^2}{(r_i+r_j)^2}.
\end{align*}
Combining the results of the preceding two displays yields the desired formula.
\end{proof}

\subsection{Instantaneous separation and positivity of the roots}

\begin{theorem} \label{thm:root-separation}
    If $p(x)$ is a monic polynomial of degree $d$ with nonnegative real roots, then the roots of $p(x,t) = (p \boxplus^n_d G_{n,d}^{[t]})(x)$ are positive and distinct for all $t > 0$.
\end{theorem}

\begin{proof}
    We first show that if the roots of $p(x,t)$ are positive and distinct at some time $\tau > 0$, then they remain positive and distinct at all times $t \ge \tau$.  The assumption at time $\tau$ implies that $H(r(\tau)) > - \infty$. If
    \[
    \tau' = \inf \{ \ t \ge \tau \ : \ p(x,t) \text{ has a multiple root or a zero root at time $t$} \ \} < \infty,
    \]
    then $H(r(\tau')) = - \infty$, implying that $\partial_t H(r(t)) < 0$ at some time $t \in (\tau, \tau')$, which contradicts the result of Corollary \ref{cor:dtH} that $H(r(t))$ is strictly increasing at all times when the roots are positive and distinct.  Therefore, the roots remain positive and distinct at all times $t \ge \tau$.

    Thus, it suffices to show that if the roots of $p(x,t)$ are nonnegative at $t = 0$, then they are strictly positive and distinct for all sufficiently small $t > 0$.
    
    First, let us prove the strict positivity of the roots.
    More explicitly, we will prove that all the roots of $p(x,t)$, as a polynomial in $x$, are strictly positive whenever $t>0$ is small.
    Indeed, let us assume the opposite.
    From \eqref{eqn:pxt-coeffs}, we can write $p(x,t)=\sum_{i=0}^d{c_i(t)x^i}$, where each $c_i(t)$ is a polynomial in $t$.
    Since the roots of $p(x,t)$ are nonnegative, our assumption implies that $0$ is a root of $p(x,t)$ for all small $t>0$.
    It follows that $c_0(t)=0$ for infinitely many positive values of $t$ approaching zero, so $c_0(t)$ must be the zero polynomial.
    However, from \eqref{eqn:pxt-coeffs} we have $n!\,c_0(t)=\sum_{j=0}^d (n+d-j)!\,a_{2j}(-2t)^{d-j}$, where $a_0=1$; in particular, $c_0(t)$ is not identically zero.
    This contradiction completes the proof of the claim.
        
    Next, we verify the simple-rootedness of $p(x,t)$, as a polynomial in $x$, for all small $t>0$.
    If $p(x,0)=p(x)$ has all simple roots, then so does $p(x,t)$ for small $t>0$, because the roots are continuous functions of the coefficients of the polynomial, and these coefficients are continuous (polynomial) functions of $t$.
    So assume that not all roots of $p(x)$ are simple.
    The rest of the proof will invoke the PDE (\ref{eqn:pde}).

    Let us first consider the simplest case, when $p(x)$ has a double root at some $x_0 > 0$ and its other roots are simple.
    Since $p(x_0,0) = \partial_x p(x_0,0) = 0$, we have
    \[
    \partial_t p(x_0,0) = -2 x_0 \partial^2_x p(x_0,0),
    \]
    where $\partial^2_x p(x_0,0) \ne 0$ because $x_0$ is a double root.  If $\partial^2_x p(x_0,0) > 0$, so that $p(x)$ is convex at $x_0$, then $p(x_0,t)$ is decreasing and thus negative for small $t > 0$.  On the other hand, if $\partial^2_x p(x_0,0) < 0$, so that $p(x)$ is concave at $x_0$, then $p(x_0,t)$ is increasing and thus positive for small $t > 0$, so that in both cases the roots at $x_0$ separate.

    For higher-order roots, the argument is similar.
    Suppose that $p(x)$ has a multiple root of multiplicity $2j$ at $x_0 > 0$, so that $\partial_x^\ell p(x_0) = 0$ for $\ell = 0,\hdots, 2j-1$, but $\partial_x^{2j}p(x_0)\ne 0$.
    This means in particular that $\partial_x^{2j-2}p$ has a double root at $x_0$.
    Observe that
    \[
    \partial_t\big( \partial_x^{2j-2}p \big) = \partial_x^{2j-2} \big( \partial_tp \big) = \partial_x^{2j-2} \big( -2x\partial_x^2p - 2(n+1)\partial_xp \big),
    \]
    and therefore
    \[
    \partial_t\big( \partial_x^{2j-2}p \big)(x_0,0) = -2x_0\partial_x^{2j}p(x_0,0).
    \]
    With this equality, we can argue as before that regardless of the sign of $\partial_x^{2j}p(x_0, 0)$, the double root of $\partial_x^{2j-2}p(x,t)$ at $x=x_0$ and $t=0$ separates for small $t>0$.
    Equivalently, the root of multiplicity $2j$ of $p(x)$ must split into at least two roots (each of multiplicity strictly less than $2j$) for small times $t>0$.
    The argument for multiple roots of $p(x)$ with odd multiplicities is almost identical, so we omit it.

    The fact that \emph{all} roots must be distinct for $t > 0$ then follows by strong induction on the multiplicity of the roots: a root of multiplicity $k$ cannot split into lower-order roots in such a way that any root at time $t > 0$ has multiplicity greater than 1, as this would require a multiple root of lower order than $k$ to remain stable up to time $t$, contradicting the inductive hypothesis.
\end{proof}

\subsection{Linear drift of the center of mass}

Next, we identify a conserved quantity for the motion, which leads to a formula for the evolution of the center of mass of the roots.

\begin{proposition} \label{prop:conserved}
If $r_1(t) > \cdots > r_d(t) > 0$ satisfy the ODE system (\ref{eq:ode}), then
\begin{equation} \label{eq:virial}
\sum_{k=1}^d{r_k(t)\dot{r}_k(t)} = d(n+d).
\end{equation}
Consequently,
\begin{equation} \label{eq:moment-growth}
\sum_{k=1}^d{r_k(t)^2} = \sum_{k=1}^d{r_k(0)^2} + 2d(n+d)t
\end{equation}
for all $t\ge 0$, implying that the first moment of the root distribution of $p(x,t)$ grows linearly in time.
\end{proposition}

\begin{proof}
Multiplying both sides of (\ref{eq:ode}) by $r_k$ and summing over $k$ gives, after simplification,
\begin{equation} \label{eq:virial-expand}
\sum_{k=1}^d r_k \dot{r}_k = 2\cdot\sum_{\substack{1\le k,j\le d\\ k\ne j}} \frac{r_k^2}{r_k^2-r_j^2} + (n+1)d.
\end{equation}
Swapping indices $k$ and $j$ and performing simple algebraic manipulations, we find
\begin{align*}
\sum_{\substack{1\le k,j\le d\\ k\ne j}} \frac{r_k^2}{r_k^2-r_j^2}
&= \sum_{\substack{1\le k,j\le d\\ k\ne j}} \frac{r_j^2}{r_j^2-r_k^2}
= \sum_{\substack{1\le k,j\le d\\ k\ne j}} \left( 1 - \frac{r_k^2}{r_k^2-r_j^2} \right)\\
&= d(d-1) - \sum_{\substack{1\le k,j\le d\\ k\ne j}} \frac{r_k^2}{r_k^2-r_j^2},
\end{align*}
which implies that
\[
\sum_{k \neq j} \frac{r_k^2}{r_k^2-r_j^2} = \frac{d(d-1)}{2}.
\]
Substituting into (\ref{eq:virial-expand}) then gives
\[
\sum_{k=1}^d r_k \dot{r}_k = 2\cdot\frac{d(d-1)}{2} + (n+1)d = d(n+d).
\]
Since $\frac{d}{dt}\sum_i r_i^2 = 2\sum_i r_i\dot{r}_i$, integrating (\ref{eq:virial}) yields (\ref{eq:moment-growth}).
\end{proof}

\begin{remark}
In the (non-rectangular) finite free heat flow (\ref{eqn:ffhf}), if $p(x,t) = \prod_i (x - a_i(t))$ solves the time-reversed heat equation $\partial_t p = -\frac{1}{2}\partial_x^2 p$, then $\sum_i \dot{a}_i = 0$, so the center of mass of the roots is conserved. In contrast, for the rectangular finite free heat flow, the first moment $\frac{1}{d}\sum_i{r_i^2}$, namely the average of the roots of $p(x,t)$, grows linearly with rate $2d(n+d)$.
\end{remark}

\subsection{Evolution of the Cauchy transform}

We now establish a generalized Burgers equation for the Cauchy transform of the root distribution.  The following proposition is analogous to the relationship between the classical heat equation and the Burgers equation via the Cole--Hopf transformation \cite{cole1951, hopf1950}, as well as to the fact the Cauchy transform of the free heat flow (\ref{eqn:freehf}) satisfies the complex Burgers equation (\ref{eqn:cx-burgers}).

Recall that the \emph{Cauchy transform} of a probability measure $\mu$ on $\R$ is the function
\[
G_\mu(x) = \int_{\R} \frac{1}{x-y} \,\mu(\mathrm{d}y), \qquad x \in \C \setminus \mathrm{supp}(\mu).
\]
The Cauchy transform is sometimes also referred to as the Stieltjes transform and defined with the opposite sign.  For the empirical distribution of the roots $r_1(t)^2, \hdots, r_d(t)^2$ of the polynomial $p(x,t)$, the Cauchy transform is
\[
G(x,t)
\;=\; \frac{1}{d} \sum_{i=1}^d \frac{1}{x - r_i(t)^2}.
\]
In this case, the Cauchy transform is just a constant multiple of the logarithmic derivative of $p(x,t)$:
\begin{equation} \label{eqn:G-logd}
G(x,t) = \frac{1}{d} \frac{\partial_x p(x,t)}{p(x,t)}.
\end{equation}
The above observation justifies the analogy with the Cole--Hopf transformation, which describes solutions of the classical viscous Burgers equation as logarithmic derivatives of solutions of the heat equation.

\begin{proposition} \label{prop:cauchy_pde}
Let $p(x,t) = \prod_{i=1}^d (x - r_i(t)^2)$ be a polynomial of degree $d$ with nonnegative real roots that satisfies the partial differential equation
\begin{equation} \label{eq:poly_pde}
{\partial_t} p(x,t) = -2\left( x {\partial_x^2} + (n+1) {\partial_x} \right) p(x,t).
\end{equation}
Then the Cauchy transform of the empirical root distribution \[
G(x,t) = \frac{1}{d}\sum_{i=1}^d \frac{1}{x - r_i(t)^2}
\]
satisfies the partial differential equation
\begin{equation} \label{eq:cauchy_pde}
\partial_t G + 2 \partial_x \left( d x G^2 + (n+1) G + x \partial_x G \right) = 0.
\end{equation}
\end{proposition}
\begin{proof}
Differentiating (\ref{eqn:G-logd}) with respect to $t$, we have:
\begin{equation} \label{eq:dtG}
\partial_t G = \frac{1}{d}\partial_t \left( \frac{\partial_x p}{p} \right) = \frac{1}{d}\partial_x \left( \frac{\partial_t p}{p} \right).
\end{equation}
The equation (\ref{eq:poly_pde}) for $\partial_t p$ gives
\begin{equation}\label{partial_t}
\frac{\partial_t p}{p} = \frac{-2x \partial_x^2 p - 2(n+1) \partial_x p}{p} = -2x \left( \frac{\partial_x^2 p}{p} \right) - 2(n+1) \frac{\partial_x p}{p},
\end{equation}
while the identity for the second logarithmic derivative,
\[ \partial_x \left(\frac{\partial_x p}{p}\right) = \frac{\partial_x^2 p}{p} - \left(\frac{\partial_x p}{p}\right)^2,
\]
gives
\[ \frac{\partial_x^2 p}{p} = d\,\partial_x G + d^2 G^2. \]
Substituting this and $\partial_x p/p = dG$ back into \eqref{partial_t}, we have:
\[ \frac{\partial_t p}{p} = -2x(d\,\partial_x G + d^2 G^2) - 2d(n+1)G. \]
Plugging the above into (\ref{eq:dtG}) yields
\[ \partial_t G = \frac{1}{d} \partial_x \left[ - 2dx \partial_x G - 2d^2 x G^2 - 2d(n+1)G \right], \]
and rearranging terms gives the desired PDE (\ref{eq:cauchy_pde}).
\end{proof}

\begin{remark}
    We have referred to the PDE (\ref{eq:cauchy_pde}) as a generalized Burgers equation due to the analogy with the Cole--Hopf transformation: it is a nonlinear equation that describes the logarithmic derivative of a solution to the linear equation (\ref{eq:poly_pde}). However, the comparison should be taken lightly, as (\ref{eq:cauchy_pde}) is qualitatively very different from the classical Burgers equation
    \begin{equation} \label{eqn:viscous-burgers}
    \partial_t u + u\partial_x u = \nu \partial_x^2 u.
    \end{equation}
    Distributing the outer spatial derivative, the PDE (\ref{eq:cauchy_pde}) can be rewritten as
    \begin{equation}
        \partial_tG + 2dG^2 + 4dx G\partial_x G + 2(n+2) \partial_x G + 2 x \partial_x^2 G = 0.
    \end{equation}
    The above equation features a quadratic term $2dG^2$ and a linear drift term $2(n+2) \partial_x G$ that have no analogues in (\ref{eqn:viscous-burgers}); moreover, the nonlinear advection term $4dx G \partial_x G$ and the second-order term $2x \partial_x^2 G$ have variable coefficients that vanish as $x \to 0$, making the equation degenerate at the origin.
\end{remark}

\section{Asymptotics}
\label{sec:asymptotics}

In this section, we study the asymptotic behavior of the rectangular finite free heat flow in two different limit regimes: the long-time limit $t \to \infty$ and the high-degree limit $d \to \infty$.  In the first setting, we prove an analogue of the fact that the Gaussian distribution is the unique fixed point of the classical heat flow after rescaling space by $\sqrt{t}$.  In the second setting, we show that the high-degree limit of the rectangular finite free heat flow is described by the rectangular free convolution with a Marchenko--Pastur distribution.

\subsection{Long-time asymptotics}
\label{sec:long-time}

Here we consider the asymptotic distribution of the roots of the rectangular finite free heat flow at large times $t$.
As $t\to\infty$, we prove that the rescaled roots $r_k(t)^2/t$ approach the corresponding roots of the polynomial $G^{[1]}_{n,d}(x)$, which are twice the roots of the Laguerre polynomial $L_d^{(n)}(x)$.  This result does not depend on the initial data and is analogous to the fact that under the classical heat flow, and after rescaling space by $\sqrt{t}$, any probability measure with finite variance will converge to a standard Gaussian as $t \to \infty$.  The proof follows from a fixed-point argument for a modified version of the entropy functional $H$.

\begin{theorem} \label{thm:long-time}
    If $p(x)$ is a monic polynomial of degree $d$ with nonnegative real roots, then the roots $r_1(t)^2 \ge \hdots \ge r_d(t)^2 \ge 0$ of $p(x,t) = (p \boxplus^n_d G_{n,d}^{[t]})(x)$ satisfy
    \begin{equation}
    \lim_{t \to \infty} \frac{r_k(t)^2}{t} \;=\; 2 \xi_k, \qquad k = 1, \hdots, d,
    \end{equation}
    where $\xi_1 > \hdots > \xi_d > 0$ are the roots of the Laguerre polynomial $L_d^{(n)}(x)$.
\end{theorem}

\begin{proof}
    First of all, note that by Theorem \ref{thm:root-separation}, we have $r_1(t) > \hdots > r_d(t) > 0$ for all $t > 0$.
    It will be convenient to define the rescaled positions
    \[
    s_k(t) = \frac{r_k(t)}{\sqrt{t}}, \qquad k = 1, \hdots, d,
    \]
    which are also positive and ordered: $s_1(t)>\dots>s_d(t)>0$.  Observe that this is indeed the correct scaling, since by the moment formula (\ref{eq:moment-growth}), we have
    \begin{equation} \label{eq:moment-lim}
    \sum_{i=1}^d s_i(t)^2 = \frac{1}{t} \sum_{i=1}^d r_i(0)^2 + 2d(n+d) \ \to \ 2d(n+d) \qquad \text{ as } t \to \infty,
    \end{equation}
    implying that $s(t) = (s_1(t), \dots, s_d(t)) \in \R^d$ remains in a bounded set for all $t>0$ without collapsing to zero.
    Next, recall that by Theorem \ref{thm:roots-ode} and Proposition \ref{prop:ode-grad}, the unscaled $r_k$ evolve according to the gradient flow $\dot r = \nabla H(r)$. The rescaled positions then satisfy the ODEs
    \begin{align} \begin{split} \label{eq:odes-rescaled}
    \dot s_k = \frac{\dot r_k}{\sqrt{t}} - \frac{r_k}{2t^{3/2}} &= \frac{1}{\sqrt{t}} \left[ \sum_{j\colon j \ne k} \left( \frac{1}{r_k - r_j} + \frac{1}{r_k + r_j} \right) + \frac{n+1}{r_k} \right] - \frac{s_k}{2t} \\
    &= \frac{1}{t} \left[ \sum_{j\colon j \ne k} \left( \frac{1}{s_k - s_j} + \frac{1}{s_k + s_j} \right) + \frac{n+1}{s_k} \right] - \frac{s_k}{2t}.
    \end{split} \end{align}
    At a fixed point of the rescaled system, where $\dot s_k = 0$ for all $k=1,\dots, d$, we therefore need
    \[
    \sum_{j : j \ne k} \left( \frac{1}{s_k - s_j} + \frac{1}{s_k + s_j} \right) + \frac{n+1}{s_k} = \frac{s_k}{2}, \qquad k = 1, \dots, d.
    \]
    Using the identity 
    \[
    \frac{1}{s_k - s_j} + \frac{1}{s_k + s_j} = \frac{2s_k}{s_k^2 - s_j^2}
    \]
    and writing $u_k = s_k^2$, we can divide through by $s_k > 0$ to rewrite the fixed-point equations as
    \begin{equation} \label{eq:fp-u}
    \sum_{j : j \ne k} \frac{2}{u_k - u_j} + \frac{n+1}{u_k} = \frac{1}{2}, \qquad k = 1, \hdots, d.
    \end{equation}
    Note that these new variables are also positive and distinct, that is, $u_1 > \hdots > u_d > 0$. Now let $Q(x) = \prod_{j=1}^d (x - u_j)$ be the monic polynomial with roots $u_1,\dots,u_d$. Since
    \[
    \frac{Q''(u_k)}{Q'(u_k)} = \sum_{j : j \ne k} \frac{2}{u_k - u_j},
    \]
    the fixed-point equation (\ref{eq:fp-u}) can be written as
    \[
    u_k Q''(u_k) + (n+1) Q'(u_k) = \frac{u_k}{2} Q'(u_k), \qquad k = 1, \hdots, d.
    \]
    This says that the polynomial $xQ'' + (n+1 - x/2)Q'$ vanishes at each root of $Q$ and is therefore divisible by $Q$. But $xQ'' + (n+1-x/2)Q'$ and $Q$ are both polynomials of the same degree $d$, so in fact
    \begin{equation} \label{eq:Q-ode-lambda}
    x Q''(x) + \left(n+1 - \frac{x}{2}\right) Q'(x) + \lambda Q(x) = 0
    \end{equation}
    for some constant $\lambda$. 
    Since $Q$ is monic of degree $d$, the terms $xQ''$ and $(n+1)Q'$ contribute only up to 
    degree $d-1$, while $(-x/2)Q'(x)$ and $\lambda Q(x)$ each contribute a term of degree $d$, with leading coefficients $-d/2$ and $\lambda$ respectively. For (\ref{eq:Q-ode-lambda}) to hold, the coefficient of $x^d$ must vanish, giving $\lambda = d/2$, so $Q$ satisfies
    \[
    x Q''(x) + \left(n+1 - \frac{x}{2}\right) Q'(x) + \frac{d}{2} Q(x) = 0.
    \]
    Substituting $Q(x) = \tilde{Q}(x/2)$ and using $Q'(x) = \frac{1}{2}\tilde{Q}'(x/2)$ 
    and $Q''(x) = \frac{1}{4}\tilde{Q}''(x/2)$, then replacing $x/2$ by $x$ throughout, 
    we find that $\tilde{Q}$ satisfies the Laguerre ODE
    \begin{equation} \label{eq:lag-ode}
    x \tilde{Q}''(x) + (n+1-x)\tilde{Q}'(x) + d\cdot\tilde{Q}(x) = 0.
    \end{equation}
    The unique monic degree-$d$ polynomial solution to this ODE is $(-1)^d d! \cdot 
    L_d^{(n)}(x)$, so $Q(x) = (-1)^d d! \cdot L_d^{(n)}(x/2)$, and the roots of $Q(x)$ 
    are $u_k = 2\xi_k$, where $\xi_1 > \dots > \xi_d > 0$ are the roots of $L_d^{(n)}(x)$.

    So far we have shown that the rescaled ODE system (\ref{eq:odes-rescaled}) has a unique fixed point $s^*=(s_1^*,\dots,s_d^*)$ with $s_1^*>\dots>s_d^*>0$, given by $s_k^* = \sqrt{2\xi_k}$ for all $k=1,\dots,d$. It remains to show that the vector of rescaled positions $s(t)=(s_1(t),\dots,s_d(t))$ actually converges to this fixed point.
    
    The ODE system for $s(t)$ can be written as
    \[
    \dot{s} = \frac{1}{t} \nabla \tilde{H}(s),
    \]
    where $\tilde{H}(s)$ is the modified entropy:
    \[
    \tilde H(s) = H(s) - \frac{1}{4} \sum_{i=1}^d s_i^2 = \log \left[ \left( \prod_{i=1}^d s_i^{n+1} \right) \prod_{1 \le i<j \le d} (s_i^2 - s_j^2) \right] - \frac{1}{4} \sum_{i=1}^d s_i^2.
    \]
    Introducing the time change $\tau = \log t$ transforms this into a genuine gradient flow,
    \[
    \frac{ds}{d\tau} = \nabla \tilde{H}(s).
    \]
    
    We now verify the hypotheses needed to conclude convergence. First, as pointed out before, the vector $s(\tau)$ remains in the region
    \[
    \Omega = \{(s_1,\dots,s_d) \ | \ s_1 > \hdots > s_d > 0\}
    \]
    for all $\tau\in\mathbb{R}$. Second, starting from any finite $\tau_0 \in \R$, the trajectory $\{ s(\tau) \ | \ \tau \ge \tau_0 \}$ is in fact confined to a compact subset of $\Omega$. To see this, note that $\tilde H(s) \to -\infty$ both as $s$ approaches $\partial\Omega$ and as $|s| \to \infty$, so every superlevel set $\{s \in \Omega \ | \ \tilde H(s) \ge C\}$ is a compact subset of $\Omega$. Since $\tilde H(s(\tau))$ is nondecreasing along the flow, the trajectory remains in the compact superlevel set $\{\tilde H \ge \tilde H(s(\tau_0))\}$ for all $\tau \ge \tau_0$. Third, $\tilde{H}$ is strictly concave on $\Omega$: indeed, writing
    \[
    \tilde{H}(s) = \sum_{i=1}^d (n+1)\log s_i + \sum_{1 \le i < j \le d} \log(s_i - s_j) 
    + \sum_{1 \le i < j \le d} \log(s_i + s_j) - \frac{1}{4}\sum_{i=1}^d s_i^2,
    \]
    we observe that the first three terms are concave functions of linear forms in $s$, hence they have negative semidefinite Hessians; their sum is therefore also negative semidefinite. The fourth term $-\frac{1}{4}\sum_i s_i^2$ contributes a further $-\frac{1}{2}I$ to the Hessian, so $D^2\tilde{H}$ is negative definite on $\Omega$.
    As a result, the critical point $s^* \in \Omega$ identified above is unique and is a global maximum.

    Since $\tilde{H}(s(\tau))$ is nondecreasing along the flow and the trajectory is eventually confined to a compact subset of $\Omega$, LaSalle's invariance principle \cite{lasalle1961stability} implies that $s(\tau)$ converges as $\tau\to\infty$ to some point in the set of critical points of $\tilde{H}$ in $\Omega$. Since $s^*$ is the only such point, we have that $s(\tau) \to s^*$, or equivalently, $r_k(t)^2/t \to 2\xi_k$ as $t\to\infty$.
\end{proof}

\begin{remark}
    The rate of convergence in Theorem \ref{thm:long-time} can be made explicit via a local analysis of the gradient flow $ds/d\tau = \nabla\tilde{H}(s)$ near the fixed 
    point $s^*$. Since $D^2\tilde{H}(s^*)$ is negative definite, $s^*$ is a nondegenerate critical point of $\tilde{H}$, and small perturbations $\delta s = s - s^*$ decay exponentially in $\tau$ at a rate governed by the largest eigenvalue $\lambda_{\max} 
    < 0$ of $D^2\tilde{H}(s^*)$. Substituting $\tau = \log t$, this gives
    \begin{equation}
    \left|\frac{r_k(t)^2}{t} - 2\xi_k\right| = O(t^{\lambda_{\max}})
    \end{equation}
    for large times $t$. In principle, $\lambda_{\max}$ can be computed in terms of the Laguerre zeros $\xi_1, \hdots, \xi_d$, though it is not obvious whether it admits an explicit formula. Regardless, there is a straightforward upper bound: since $D^2\tilde{H}(s^*) = D^2H(s^*) - \frac{1}{2}I$ and $D^2H(s^*)$ is negative semidefinite, we have $\lambda_{\max} < -1/2$, so the convergence is at least as fast as $O(t^{-1/2})$.
\end{remark}

\begin{remark}
The appearance of the Laguerre polynomial $L_d^{(n)}$ in Theorem \ref{thm:long-time} has a probabilistic interpretation in terms of the squared Bessel process discussed in Remark \ref{rem:bessel}. The Laguerre differential equation (\ref{eq:lag-ode}) says that $L_d^{(n)}$ is an eigenfunction of the Laguerre operator 
\[
\mathcal{L}_{\mathrm{Lag}}^{(n)} = x\partial_x^2 + (n+1-x)\partial_x
\]
with eigenvalue $-d$.  The Laguerre operator differs from the squared-Bessel generator $\mathcal{L} = 2x\partial_x^2 + 2(n+1)\partial_x$ (up to a factor of $2$) by the addition of a linear restoring drift $-x\partial_x$, which corresponds to a quadratic confining potential at the origin.  The rescaling $s_k(t) = r_k(t)/\sqrt{t}$ in the proof of Theorem \ref{thm:long-time} introduces precisely this drift, and due to the resulting confinement, the rescaled squared Bessel process has a stationary measure given by a Gamma distribution. Theorem \ref{thm:long-time} says that after the same rescaling, the rectangular finite free heat flow converges to $G^{[1]}_{n,d}(x) \propto L_d^{(n)}(x/2)$; as $d$ ranges over the natural numbers, these are precisely the orthogonal polynomials for the stationary measure.
\end{remark}

\begin{remark}
The limiting positions $2\xi_k$ identified in Theorem \ref{thm:long-time} are consistent with results of Andraus--Katori--Miyashita \cite{AKM2} and Andraus--Miyashita \cite{Andraus-Miyashita}, who proved analogous convergence theorems for a family of stochastic interacting particle systems in the so-called \emph{freezing limit}; see Section \ref{sec:dunklprocs} for further discussion of the connection between these stochastic processes and the rectangular finite free heat flow. While these results are closely related to Theorem \ref{thm:long-time}, to the best of our knowledge the deterministic statement and convergence argument above are new.
\end{remark}

\subsection{High-degree asymptotics}
\label{sec:high-degree}

As the degree $d$ grows, the rectangular finite free heat flow converges to its infinite-dimensional analogue: the rectangular free convolution with a Marchenko--Pastur distribution. In this subsection, we make this statement precise. The result, Theorem \ref{thm:high-degree} below, is a rectangular analogue of the fact that the finite free heat flow (\ref{eqn:ffhf}) converges to the free convolution with the semicircle distribution as $d\to\infty$ \cite[Theorem 2.13]{kabluchko2025lee}.  An essentially equivalent statement to our Theorem~\ref{thm:high-degree} was shown by a different argument via a cumulant method in \cite[Main Result III]{Cuenca2024}, but for the sake of clarity and completeness we give a self-contained proof here.

We briefly recall the necessary background. For $\lambda \in [0,1]$, Benaych-Georges \cite{benaych2009rectangular} defined the
\emph{rectangular free convolution} $\mu \boxplus_\lambda \nu$ of symmetric compactly supported probability
measures on $\mathbb{R}$, arising as the limit of symmetrized empirical singular value distributions
of sums of independent rectangular random matrices.  The map $x \mapsto x^2$ gives a bijective correspondence between symmetric measures on $\mathbb{R}$ and measures on $\mathbb{R}_{\ge 0}$; via this correspondence, the convolution $\boxplus_\lambda$
induces an analogous convolution $\boxplus_\lambda^+$ on compactly supported probability measures on $\mathbb{R}_{\ge 0}$, which describes the limiting empirical squared singular value distributions.  More precisely, if $A_d$ and $B_d$ are independent
$d \times n$ random matrices, which are unitarily bi-invariant and
whose empirical squared singular value distributions converge weakly to $\mu$ and $\nu$
respectively as $d \to \infty$ with $d/n \to \lambda$, then the empirical squared singular value distribution of $A_d + B_d $ converges weakly to $\mu \boxplus^+_\lambda \nu$ almost surely.

To state our convergence result for the rectangular finite free heat flow, recall that in the definition (\ref{eqn:rectffconv}) of the convolution $\boxplus_d^n$, the roots of the polynomials correspond to squared singular values of matrices. We associate to
each monic degree-$d$ polynomial $p$ with nonnegative real roots
$\rho_1, \ldots, \rho_d \geq 0$, its \emph{scaled} empirical root distribution:
\[
    \mu_p = \frac{1}{d}\sum_{k=1}^d \delta_{\rho_k/d}.
\]
We also recall from \cite{KornyikMichaletzky} that for $c > 0$, the
\emph{Marchenko--Pastur distribution}\footnote{The Marchenko--Pastur distribution is sometimes defined slightly differently in the literature; see Remark~\ref{rem:MP}.} $\mathrm{MP}_c$ is the probability measure on $\mathbb{R}_{\ge 0}$ with density
\[
    \frac{\sqrt{(x_+ - x)(x - x_-)}}{2\pi x}\,\mathbf{1}_{[x_-,\, x_+]}(x),
    \qquad x_\pm = \bigl(\sqrt{c+1} \pm 1\bigr)^2,
\]
and that $\frac{1}{d}\sum_{k=1}^d \delta_{\xi_k/d} \to \mathrm{MP}_{c}$ weakly as
$d \to \infty$ with $n/d \to c$, where $\xi_1,\ldots,\xi_d$ are the roots of
$L_d^{(n)}$. We write $\mathrm{MP}_{c,t}$ for the pushforward of $\mathrm{MP}_c$ under
$x \mapsto tx$.

\begin{theorem}\label{thm:high-degree}
Let $(p_d)_{d \geq 1}$ be a sequence of monic polynomials with nonnegative real roots, where each $p_d$ has degree $d$, and suppose that the scaled empirical root distributions $\mu_{p_d}$ converge weakly to a compactly supported probability measure $\mu$ on $\mathbb{R}_{\ge 0}$ as $d\to\infty$.
Let $n = n(d)$ be a sequence of positive integers with $d/n \to\lambda\in (0,1]$. Then for each $t \geq 0$,
\[
    \mu_{p_d \,\boxplus^n_d\, G_{n,d}^{[t]}}
    \xrightarrow{\, \ \, } \mu \boxplus_\lambda^+ \mathrm{MP}_{1/\lambda,\, 2t}
\]
weakly as $d\to\infty$.
\end{theorem}

\begin{proof}
The argument reduces to two ingredients: determining the asymptotic root distribution of $G_{n,d}^{[t]}$, and determining the high-degree behavior of the rectangular convolution $\boxplus^n_d$.

For the asymptotics of $G_{n,d}^{[t]}$, we claim that $\mu_{G_{n,d}^{[t]}} \to \mathrm{MP}_{1/\lambda,\,2t}$ weakly.
By definition \eqref{eqn:Gdef}, we have $G_{n,d}^{[t]}(x) = (-2t)^d \cdot d! \cdot L_d^{(n)}(x/2t)$, so the roots of $G_{n,d}^{[t]}(x)$ are $2t\xi_1, \ldots, 2t\xi_d$, where $\xi_1, \ldots, \xi_d$ are the roots of $L_d^{(n)}(x)$.
As a result, $\mu_{G_{n,d}^{[t]}}$ is the pushforward of $\frac{1}{d}\sum_{k=1}^d \delta_{\xi_k/d}$ under the map $x\mapsto 2tx$.
Then \cite[Theorem~1]{KornyikMichaletzky} shows that, in the regime where $d\to\infty$, $n/d \to 1/\lambda$, the scaled empirical
measure $\frac{1}{d}\sum_{k=1}^d \delta_{\xi_k/d}$ converges weakly to
$\mathrm{MP}_{1/\lambda}$, and hence $\mu_{G_{n,d}^{[t]}} \to \mathrm{MP}_{1/\lambda,\,2t}$ in the same limit regime.

Then, by a result of Gribinski
\cite[Theorem 6.1]{gribinski2024theory},\footnote{The result \cite[Theorem 6.1]{gribinski2024theory} actually gives a stronger mode of convergence and is stated in terms of the symmetrized root distributions and the rectangular free convolution $\boxplus_\lambda$. Pushing the measures forward by the map $x \mapsto x^2$ results in the analogous statement for the squared root distributions and $\boxplus^+_\lambda$, which is what we use here.} the rectangular finite free convolution $\boxplus^n_d$
converges to $\boxplus_\lambda^+$ in the regime where $d\to\infty$, $d/n\to \lambda$: if $\mu_{q_d} \to \mu$ and $\mu_{s_d} \to \nu$ weakly, for sequences of polynomials $q_d, s_d$ of degree $d$ with nonnegative roots, then $\mu_{q_d \boxplus^n_d s_d} \to \mu \boxplus_\lambda^+ \nu$ weakly. Applying this for $q_d = p_d$ and $s_d = G_{n,d}^{[t]}$, together with the above limit $\mu_{G_{n,d}^{[t]}} \to \mathrm{MP}_{1/\lambda,\,2t}$, completes the proof.
\end{proof}

\begin{remark}
The measure $\mu \boxplus_\lambda^+ \mathrm{MP}_{1/\lambda,\,2t}$ can be regarded as the distribution at time $t$ of a \emph{rectangular free heat flow}: the infinite-dimensional evolution obtained by replacing the finite convolution $\boxplus^n_d$ and the heat polynomials $G^{[t]}_{n,d}$ with their large-$d$ limits $\boxplus_\lambda^+$ and $\mathrm{MP}_{1/\lambda,\,2t}$. Theorem~\ref{thm:high-degree}
thus establishes that the rectangular finite free heat flow converges, in the
high-degree limit, to this evolution on probability measures.
\end{remark}

\begin{remark}\label{rem:MP}
Following \cite{KornyikMichaletzky}, we defined the Marchenko--Pastur distribution $\mathrm{MP}_c$ in Theorem~\ref{thm:high-degree} in such a way that the roots $\xi_k$ of $L_d^{(n)}$
satisfy $\frac{1}{d}\sum_k \delta_{\xi_k/d} \to\mathrm{MP}_{1/\nu}$ as $d\to\infty$ and $d/n\to\nu\in(0,1]$.
This differs from the standard convention used in random matrix theory, in which the Marchenko--Pastur distribution with parameter $1/\nu$ has an atom of mass $1-\nu$ at the origin.  Up to normalization, the measure that we use in our text is the absolutely continuous part of the standard Marchenko--Pastur distribution.  See \cite[Remark 2(3)]{KornyikMichaletzky} for further details on both conventions and a precise statement of the relation between them.
\end{remark}

\section{Further perspectives} \label{sec:further}
\subsection{Calogero--Moser systems}

The ODE system (\ref{eq:ode}) describing the evolution of the roots is closely related to a certain \emph{Calogero--Moser system}. The Calogero--Moser systems are a family of completely integrable Hamiltonian systems that appear in many contexts in mathematical physics; see \cite{OlshanetskyPerelomov1981} for a review of their construction and solution methods.  In particular, (\ref{eq:ode}) can be regarded as a first-order analogue of a Calogero--Moser system of type $BC$, in the following sense.

\begin{theorem} \label{thm:1o-CM}
    Suppose that $r(t) = \big(r_1(t) > \hdots > r_d(t) > 0\big)$ satisfies
    \begin{equation}\label{eq:1o-CM}
\dot{r}_k(t) = \gamma_1 \sum_{j\colon j\ne k}{ \left(\frac{1}{r_k(t)-r_j(t)} + \frac{1}{r_k(t)+r_j(t)}\right) } + \frac{\gamma_2}{r_k(t)}
\end{equation}
for all $t \ge 0$ and some $\gamma_1, \gamma_2 \ge 0$.  Then $r(t)$ also satisfies the \emph{Calogero--Moser equations of type $BC$} with an attractive potential,
\begin{equation} \label{eq:CM-BC}
    \ddot{r}_k(t) = - 2\gamma_1^2 \sum_{j\colon j\ne k}{ \left(\frac{1}{\big(r_k(t)-r_j(t) \big)^3} + \frac{1}{\big( r_k(t)+r_j(t) \big)^3}\right) } - \frac{\gamma_2^2}{r_k(t)^3}.
\end{equation}
Conversely, any solution to (\ref{eq:CM-BC}) that satisfies (\ref{eq:1o-CM}) at $t = 0$ satisfies (\ref{eq:1o-CM}) for all $t \ge 0$.
\end{theorem}

A similar relation between the ODEs describing the finite free heat flow (\ref{eqn:ffhf}) and the Calogero--Moser system of type $A$ is well known; see \cite{calogero2001classical, sokal2023motion}. Note that although the interaction between particles in the first-order system (\ref{eq:1o-CM}) is repulsive, the interaction in the second-order system (\ref{eq:CM-BC}) is attractive.  This apparent contradiction is resolved by observing that the first-order system at $t=0$ imposes a choice of initial velocities for the second-order system that results in the particles continuing to separate despite the attractive force.

\begin{proof}
    We show that (\ref{eq:1o-CM}) implies (\ref{eq:CM-BC}).  The converse statement then follows by observing that integrating both sides of (\ref{eq:CM-BC}) recovers (\ref{eq:1o-CM}) up to an unknown additive constant, which is determined by fixing $\dot{r}_k(0)$.
    
    The equations (\ref{eq:1o-CM}) can be written
    \begin{equation} \label{eq:1CM-rewrite}
    \dot{r} = \nabla H(r; \gamma_1, \gamma_2),
    \end{equation}
    where $\nabla$ indicates the gradient with respect to $(r_1, \hdots, r_d)$, and
    \[
    H(r; \gamma_1, \gamma_2) = \log \left[ \left( \prod_{i=1}^d r_i^{\gamma_2} \right) \prod_{1 \le i<j \le d} (r_i - r_j)^{\gamma_1} (r_i + r_j)^{\gamma_1} \right]. 
    \]
    On the other hand, the equations (\ref{eq:CM-BC}) can be written
    \[
    \ddot{r} = - \frac{1}{2} \nabla \Delta H (r; \gamma_1^2, \gamma_2^2 ),
    \]
    where $\Delta = \partial_{r_1}^2 + \cdots + \partial_{r_d}^2$ is the Laplacian.

    Differentiating (\ref{eq:1CM-rewrite}) by $t$ on both sides, we obtain
    \begin{align*}
    \ddot{r} &= \big[ D^2 H(r;\gamma_1, \gamma_2) \big] \dot{r} = \big[ D^2 H(r;\gamma_1, \gamma_2) \big] \nabla H(r; \gamma_1, \gamma_2) \\
    &= \frac{1}{2} \nabla \| \nabla H(r;\gamma_1, \gamma_2) \|^2,
    \end{align*}
    where $D^2$ is the Hessian in the variables $(r_1, \hdots, r_d)$. Thus, it remains to prove that
    \begin{equation} \label{eq:fg-identity}
    \nabla \| \nabla H (r; \gamma_1, \gamma_2 ) \|^2 = - \nabla \Delta H (r; \gamma_1^2, \gamma_2^2 ).
    \end{equation}
To simplify the calculation, we now rewrite both sides of the equation above in a more symmetrical form that will be familiar to readers acquainted with the theory of root systems. Let $(e_1, \hdots, e_d)$ be the standard basis of $\mathbb{R}^d$, and define
    \begin{align*}
    \Phi^+ &= \{e_i \pm e_j \ | \ 1 \le i < j \le d \} \cup \{e_i \ | \ 1 \le i \le d \}.
    \end{align*}
    For $\alpha \in \Phi^+$, set $\kappa_\alpha = \gamma_2$ if $\alpha = e_i$ for some $i$, and $\kappa_\alpha = \gamma_1$ if $\alpha = e_i \pm e_j$ for some $i < j$.  We then can write
    \begin{align*}
    H(r; \gamma_1, \gamma_2) &= \sum_{\alpha \in \Phi^+} \kappa_\alpha \log(\alpha \cdot r), \\
    H(r; \gamma_1^2, \gamma_2^2 ) &= \sum_{\alpha \in \Phi^+} \kappa_\alpha^2 \log(\alpha \cdot r),
    \end{align*}
    and we compute
    \begin{align}
        \nonumber - \nabla \Delta H (r; \gamma_1^2, \gamma_2^2 ) &= - 2 \sum_{\alpha \in \Phi^+} \kappa_\alpha^2 \frac{|\alpha|^2 \alpha}{(\alpha \cdot r)^3}, \\
        \nonumber \nabla \| \nabla H(r;\gamma_1, \gamma_2) \|^2 &= - 2\sum_{\alpha,\beta \in \Phi^+} \kappa_\alpha \kappa_\beta \frac{(\alpha \cdot \beta)\alpha}{(\alpha \cdot r)^2 (\beta \cdot r)} \\
        \nonumber &= - 2\sum_{\alpha \in \Phi^+} \kappa_\alpha^2 \frac{|\alpha|^2 \alpha}{(\alpha \cdot r)^3} - 2\sum_{\alpha \ne \beta} \kappa_\alpha \kappa_\beta \frac{(\alpha \cdot \beta)\alpha}{(\alpha \cdot r)^2 (\beta \cdot r)}.
    \end{align}
    Therefore it suffices to show
    \[
    \sum_{\alpha \ne \beta} \kappa_\alpha \kappa_\beta \frac{(\alpha \cdot \beta)\alpha}{(\alpha \cdot r)^2 (\beta \cdot r)} = 0.
    \]
Notice that
    \[
    \sum_{\alpha \ne \beta} \kappa_\alpha \kappa_\beta \frac{(\alpha \cdot \beta)\alpha}{(\alpha \cdot r)^2 (\beta \cdot r)} = -\frac{1}{2} \nabla \left( \sum_{\alpha \ne \beta} \kappa_\alpha \kappa_\beta \frac{\alpha \cdot \beta}{(\alpha \cdot r) (\beta \cdot r)} \right).
    \]
The sum on the right-hand side vanishes by a cancellation identity due to Dunkl \cite[Proposition 1.7(1)]{Dunkl2}, completing the proof.
\end{proof}

\subsection{Radial Dunkl processes}
\label{sec:dunklprocs}

The \emph{radial Dunkl process} of type $BC$ with parameters $k_1, k_2 \ge 0$ is a diffusion process on the cone $\{x \in \R^d \ | \ x_1 > \hdots > x_d > 0\}$ with generator
\begin{equation} \label{eq:dunkl-gen}
    \mathcal{G} = \frac{1}{2}\sum_{i=1}^d \partial_{x_i}^2 + \sum_{i=1}^d \left[ \sum_{j : j \ne i} \left( \frac{k_1}{x_i - x_j} + \frac{k_1}{x_i + x_j} \right) + \frac{k_2}{x_i} \right] \partial_{x_i}.
\end{equation}
It can be constructed as the radial part of a Dunkl process, which is a generalization of Brownian motion in which the ordinary partial derivatives in the generator are replaced by the Dunkl differential-difference operators \cite{Demni-DunklProcs, Rosler-DunklOps, RoslerVoit-Markov}. In the special case where $k_1 = 1$, $k_2 = n - d + 1/2$, for an integer $n \ge d$, the radial Dunkl process of type $BC$ coincides with the process of ordered singular values of a Brownian motion on the space of $d\times n$ complex matrices \cite{guionnet2021large}.  Radial Dunkl processes and their connections to random matrix theory, representation theory, and special functions have been extensively studied; see \cite{AnkerDunklNotes, Chib-SP, Demni-DunklProcs, GY1, GY2, HuangMcSwiggen2024, Rosler-DunklOps, RoslerVoit-Markov} and references therein.

The \emph{freezing limit} of the radial Dunkl process of type $BC$ is obtained by sending the parameters to infinity at a common rate, that is, by replacing $(k_1, k_2)$ by $(\kappa k_1, \kappa k_2)$ and taking the limit $\kappa \to \infty$.  In this limit, the diffusion term in the generator (\ref{eq:dunkl-gen}) is suppressed relative to the drift terms, and the process concentrates on the solution of the deterministic ODE system obtained by retaining only the drift:
\begin{equation} \label{eq:dunkl-freeze}
    \dot{x}_i = \sum_{j : j \ne i} \left( \frac{k_1}{x_i - x_j} + \frac{k_1}{x_i + x_j} \right) + \frac{k_2}{x_i}.
\end{equation}
In the special case with $k_1 = 1$ and $k_2 = n+1$, the system (\ref{eq:dunkl-freeze}) coincides with the root evolution ODE (\ref{eq:ode}) studied in this paper.  In other words, at the level of the root dynamics, the rectangular finite free heat flow is a freezing limit of a radial Dunkl process of type $BC$.

Freezing limits of Dunkl processes have been studied in a number of works, motivated in part by their connections to random matrix theory and to the equilibrium measures of $\beta$-ensembles in the limit $\beta \to \infty$.  Notable early contributions are due to Andraus--Katori--Miyashita \cite{AKM1, AKM2} and Andraus--Miyashita \cite{Andraus-Miyashita}; in particular, these authors proved probabilistic versions of Theorem \ref{thm:long-time} above, showing that in the freezing limit the scaled process eventually concentrates on the zeros of Hermite polynomials in the type $A$ case (corresponding to the finite free heat flow (\ref{eqn:ffhf})) and on the square roots of zeros of Laguerre polynomials in the type $BC$ case.  Central limit theorems around the frozen limit were established by Voit \cite{Voit-CLT} and Andraus--Voit \cite{Andraus-Voit-CLT}, and high-dimensional limits of the processes were expressed in terms of free convolutions by Voit--Woerner \cite{VW-limits}.

\subsection{Mean-curvature evolution of group orbits}

In this final subsection, we describe a differential-geometric interpretation of the
rectangular finite free heat flow in terms of mean-curvature
flow of a family of group orbits in the space of rectangular
matrices.  This is the rectangular analogue of a result of Huang--Inauen--Menon \cite{HuangInauenMenon2023}, who showed
that the $\beta \to \infty$ limit of Dyson Brownian motion is
equivalent to expansion by mean curvature of isospectral orbits in the space of Hermitian matrices.

Let $m \ge d$ be a positive integer and consider the space
$\mathbb{R}^{d \times m}$ of real
$d \times m$ matrices, equipped with the Frobenius inner
product $\langle A, B \rangle = \mathrm{Tr}(AB^T)$.  The
group $G = \mathrm{O}(d) \times \mathrm{O}(m)$ acts on
$\mathbb{R}^{d \times m}$ isometrically by $(U, V) \cdot A = UAV^T$,
and the orbit of a matrix $A$ under this action is
\[
\mathcal{O}_A = \{ UAV^T \ | \ U \in \mathrm{O}(d),\; V \in \mathrm{O}(m) \}.
\]
The matrices in $\mathcal{O}_A$ are precisely those with the
same singular values as $A$.  For $A$ with singular value
decomposition $A = P \Sigma Q^T$, with $P \in \mathrm{O}(d)$,
$Q \in \mathrm{O}(m)$, and
$\Sigma = \mathrm{diag}(r_1, \ldots, r_d) I_{d \times m}$
with $r_1 > \cdots > r_d > 0$, the isotropy group of
$\Sigma$ is
\[
\mathrm{Stab}(\Sigma)
  = \{ (\varepsilon,\, \mathrm{diag}(\varepsilon, W)) \ | \
       \varepsilon \in \{\pm 1\}^d,\;
       W \in \mathrm{O}(m - d) \},
\]
so $\mathcal{O}_A \cong G / \mathrm{Stab}(\Sigma)$
is a compact smooth manifold of dimension $d(m-1)$, which comes with a $G$-invariant Riemannian metric induced by the Frobenius inner product on $\mathbb{R}^{d \times m}$.

For a submanifold $N$ embedded in a Riemannian manifold $(M, g)$, the \emph{mean curvature vector} at a point $x \in N$ is the trace of the second fundamental form at $x$:
\[
\mathbf{h}(x) = \sum_\alpha II(e_\alpha, e_\alpha) \in T_x^\perp N,
\]
where $\{e_\alpha\}$ is any orthonormal basis of $T_x N$ and
$II(X,Y) = (\nabla_X Y)^\perp$ is the second fundamental form,
with $(\cdot)^\perp$ denoting projection onto the normal space $T_x^\perp N$.
In our setting, $M = \mathbb{R}^{d \times m}$ is a Euclidean
space and $N = \mathcal{O}_A$, so $\nabla$ is the flat connection and $II(X,Y)$ is
simply the component of the ambient derivative $D_X Y$ normal
to $\mathcal{O}_A$.

A calculation by Pacini \cite[p.~3350]{Pac03} shows that the mean curvature vector of $\mathcal{O}_A$ at $A$ is
\begin{equation} \label{eq:pacini-rect}
\mathbf{h}(A) = -\mathrm{grad}_A \log \mathrm{vol}(\mathcal{O}_A),
\end{equation}
where $\mathrm{grad}_A$ denotes the gradient with respect to the Frobenius metric.  The relation between mean curvature flow of $G$-orbits and the rectangular finite free heat flow is a consequence of the following fact: for particular choices of matrix dimensions, the log volume of $\mathcal{O}_A$ coincides, up to an additive constant, with the entropy $H(r)$ defined in (\ref{eqn:rent-def}).

\begin{lemma} \label{lem:orbit-vol}
For $A \in \R^{d \times m}$ with singular values $r_1 > \cdots > r_d > 0$, the Riemannian volume of $\mathcal{O}_A$ satisfies
\begin{equation} \label{eq:log-vol}
\log \mathrm{vol}(\mathcal{O}_A)
  = \sum_{1 \le i < j \le d}
      \log(r_i^2 - r_j^2)
    + (m - d) \sum_{i=1}^d \log r_i
    + C,
\end{equation}
where $C$ is a constant that does not depend on the singular values.
\end{lemma}

\begin{proof}
Let $\mathfrak{g} = \mathfrak{o}(d) \oplus \mathfrak{o}(m)$ be the Lie algebra
of $G$, where $\mathfrak{o}(d)$ denotes the space of $d\times d$ real
skew-symmetric matrices. Fix a bi-invariant metric on $G$, which determines an invariant inner product on $\mathfrak{g}$.  Since $G$ acts transitively on $\mathcal{O}_A$,
the tangent space $T_\Sigma \mathcal{O}_A$ is the image of the differential
of the map $\Psi: (U,V)\mapsto U\Sigma V^T$ at the identity $(I_{d \times d}, I_{m \times m}) \in G$. We write $\phi_\Sigma = \mathrm{d} \Psi |_{(I_{d \times d}, I_{m \times m})} : \mathfrak{g} \to \R^{d\times m}$ for this differential map, which is given explicitly by
\[
\phi_\Sigma(K,L) = K\Sigma - \Sigma L.
\]
The kernel of $\phi_\Sigma$ is the isotropy subalgebra
\[
\mathrm{Lie}(\mathrm{Stab}(\Sigma)) = \{(K,L)\in\mathfrak{g} \mid K\Sigma = \Sigma L\},
\]
so $\phi_\Sigma$ restricts to an isomorphism from $\mathrm{Lie}(\mathrm{Stab}(\Sigma))^\perp$
onto $T_\Sigma\mathcal{O}_A$. The volume of $\mathcal{O}_A$ in the Frobenius
metric is therefore proportional to $\sqrt{\det g|_\Sigma}$, where $g|_\Sigma$
is the Gram matrix of the Frobenius inner products
$\langle\phi_\Sigma(K,L), \phi_\Sigma(K',L')\rangle$ on any fixed basis of $\mathrm{Lie}(\mathrm{Stab}(\Sigma))^\perp$. The proportionality constant depends only on the chosen basis and bi-invariant metric on $G$; in particular, it is independent of the singular values $r_1, \hdots, r_d$. Since we only need $\log\mathrm{vol}(\mathcal{O}_A)$
up to an additive constant, it therefore suffices to compute $\frac{1}{2}\log\det g|_\Sigma$.

We begin by identifying $\mathrm{Lie}(\mathrm{Stab}(\Sigma))$ and an explicit
basis for its orthogonal complement. Write $L \in \mathfrak{o}(m)$ in block
form
\[
L = \begin{pmatrix} L_{11} & L_{12} \\ L_{21} & L_{22}
\end{pmatrix},
\]
with $L_{11} \in \R^{d\times d}$, $L_{22} \in \R^{(m-d)\times(m-d)}$, and $L_{12} = -L_{21}^T \in \R^{d \times (m-d)}$. The condition $K\Sigma = \Sigma L$ gives
$L_{12} = 0$ from the last $m-d$ columns, and $K_{ij} r_j = r_i (L_{11})_{ij}$
from the first $d$ columns. Using skew-symmetry of both $K$ and $L_{11}$,
the latter implies $K_{ij}(r_i^2 - r_j^2) = 0$ for all $i \neq j$, and
since the $r_i$ are distinct this requires $K_{ij} = 0$ and hence
$(L_{11})_{ij} = 0$ for all $i \neq j$. Since $K$ and $L_{11}$ are
skew-symmetric their diagonal entries vanish as well, so $K = 0$ and
$L_{11} = 0$. The remaining freedom
is $L_{22} \in \mathfrak{o}(m-d)$, giving
\[
\mathrm{Lie}(\mathrm{Stab}(\Sigma)) = \{(0, \mathrm{diag}(0, L_{22})) \mid
L_{22} \in \mathfrak{o}(m-d)\},
\]
which has dimension $\binom{m-d}{2}$. We can thus verify \begin{align*}
\dim \mathrm{Lie}(\mathrm{Stab}(\Sigma))^\perp &= \dim \mathfrak{g} - \dim \mathrm{Lie}(\mathrm{Stab}(\Sigma)) \\ &=
\binom{d}{2} + \binom{m}{2} - \binom{m-d}{2} = d(m-1),
\end{align*}
matching our earlier calculation of $\dim \mathcal{O}_A$.

We now exhibit an explicit basis of $\mathrm{Lie}(\mathrm{Stab}(\Sigma))^\perp \cong T_\Sigma \mathcal{O}_A$, consisting of the following three families of tangent vectors.

\medskip
\noindent\textit{(i) Generators of $\mathfrak{o}(d)$.}
For $1 \le i < j \le d$, let $F_{ij} = E_{ij} - E_{ji} \in
\mathfrak{o}(d)$. A direct computation gives
\[
\phi_\Sigma(F_{ij}, 0) = F_{ij}\Sigma = r_j E_{ij} - r_i E_{ji},
\]
with squared Frobenius norm $r_i^2 + r_j^2$.

\medskip
\noindent\textit{(ii) Off-diagonal generators of the upper $d \times d$ block
of $\mathfrak{o}(m)$.}
For $1 \le i < j \le d$, let $\tilde{F}_{ij} = E_{ij} - E_{ji} \in
\mathfrak{o}(m)$, where $E_{ij}$ now denotes the $m\times m$ elementary matrix. Another direct computation gives
\[
\phi_\Sigma(0, \tilde{F}_{ij}) = -\Sigma\tilde{F}_{ij} = -r_i E_{ij} + r_j E_{ji},
\]
with squared Frobenius norm $r_i^2 + r_j^2$.

Since the tangent vectors $\phi_\Sigma(F_{ij}, 0)$ and $\phi_\Sigma(0, \tilde{F}_{ij})$ from families (i) and (ii) have nonzero entries only at positions $(i,j)$ and $(j,i)$ with $i,j \le d$,
vectors in either family associated with distinct pairs $\{i,j\} \neq \{i',j'\}$ are mutually
orthogonal. For a fixed pair $(i,j)$, the cross inner product between the two
families is
\[
\langle \phi_\Sigma(F_{ij},0),\, \phi_\Sigma(0,\tilde{F}_{ij}) \rangle =
\langle r_jE_{ij} - r_iE_{ji},\, -r_iE_{ij} + r_jE_{ji}\rangle = -2r_ir_j.
\]
The $2\times 2$ Gram block for each pair $(i,j)$ is therefore
\[
\begin{pmatrix} r_i^2 + r_j^2 & -2r_ir_j \\ -2r_ir_j & r_i^2 + r_j^2
\end{pmatrix},
\]
with determinant $(r_i^2 - r_j^2)^2$.

\medskip
\noindent\textit{(iii) Generators of $\mathfrak{o}(m)$ mixing the $d$ and
$(m-d)$ blocks.}
For $1 \le i \le d$ and $d < j \le m$, let $G_{ij} = E_{ij} - E_{ji} \in
\mathfrak{o}(m)$. Then
\[
\phi_\Sigma(0, G_{ij}) = -\Sigma G_{ij} = -r_i E_{ij},
\]
with squared Frobenius norm $r_i^2$. These $d(m-d)$ vectors are pairwise
orthogonal, since they are supported on distinct entries $(i,j)$ with $j > d$.
They are also orthogonal to both families (i) and (ii), since those vectors are
supported on entries with both indices at most $d$.

\medskip
The three families together comprise $\binom{d}{2} + \binom{d}{2} + d(m-d) =
d(m-1) = \dim\mathcal{O}_A$ vectors. From the above calculations, we find the Gram determinant
\[
\prod_{i<j}(r_i^2-r_j^2)^2\cdot\prod_i r_i^{2(m-d)}.\]
Since the
$r_i$ are distinct and positive, this determinant is nonzero, confirming that these vectors are linearly independent and hence form a basis of $\mathrm{Lie}(\mathrm{Stab}(\Sigma))^\perp$. The quantity above is therefore equal to $\det g|_\Sigma \propto \mathrm{vol}(\mathcal{O}_A)^2$, and taking half the logarithm gives \eqref{eq:log-vol}.
\end{proof}

The main result of this subsection is now almost immediate.

\begin{theorem} \label{thm:mcf}
Let $m = n + d + 1$.  Suppose that $A(t) \in \R^{d \times m}$ evolves by $-1$ times mean curvature on the orbit $\mathcal{O}_{A(t)}$, that is,
\begin{equation} \label{eq:mcf}
\dot A(t) = -\mathbf{h}(A(t)), \qquad t \ge 0,
\end{equation}
and that $A(0)$ has distinct positive singular values $r_1(0) > \hdots > r_d(0) > 0$. Let
\[
A(0) = P \, \mathrm{diag}\big(r_1(0), \hdots, r_d(0)\big) \, I_{d \times m} Q^T
\]
be the singular value decomposition of $A(0)$. Then
\[
A(t) = P \, \mathrm{diag}\big(r_1(t), \hdots, r_d(t)\big) \, I_{d \times m} Q^T,
\]
where the singular values $r_1(t) > \cdots > r_d(t) > 0$ of $A(t)$ satisfy the ODE system \eqref{eq:ode} describing the root evolution of the rectangular finite free heat flow.
\end{theorem}

\begin{proof}
Since $\log \mathrm{vol}(\mathcal{O}_A)$ depends only on the
singular values of $A$ by \eqref{eq:log-vol}, its gradient
$\mathbf{h}(A) = -\mathrm{grad}_A \log \mathrm{vol}(\mathcal{O}_A)$
lies in the normal space to $\mathcal{O}_A$ at $A$.  The normal
space to $\mathcal{O}_A$ at $A = P\Sigma Q^T$ consists of
matrices of the form $P D \, I_{d \times m} Q^T$ where $D$ is a real diagonal
$d \times d$ matrix; in other words, the normal directions are those that change the singular values while fixing the singular vectors.  Therefore the
evolution \eqref{eq:mcf} has no tangential component, and $P$
and $Q$ remain constant, giving
\[
A(t) = P \, \mathrm{diag}\big(r_1(t), \hdots, r_d(t)\big)
  \, I_{d \times m} Q^T
\]
for all $t \ge 0$.

It remains to show that the singular values satisfy the ODE
system \eqref{eq:ode}.  Comparing \eqref{eq:log-vol} with the
rectangular entropy \eqref{eqn:rent-def},
\[
H(r) = \sum_{i < j} \log(r_i^2 - r_j^2)
      + (n+1) \sum_{i=1}^d \log r_i,
\]
we see that $\log \mathrm{vol}(\mathcal{O}_A) = H(r) +
\mathrm{const}$ when $m = n + d + 1$.  Therefore the component
of $\mathbf{h}(A(t))$ in the direction of the $k$-th singular
value is
\[
-\frac{\partial}{\partial r_k} \log \mathrm{vol}(\mathcal{O}_A)
  = -\frac{\partial H}{\partial r_k}(r(t)),
\]
and the evolution \eqref{eq:mcf} by $-1$ times mean
curvature gives
\[
\dot{r}_k(t)
  = -\left( -\frac{\partial H}{\partial r_k}(r(t)) \right)
  = \frac{\partial H}{\partial r_k}(r(t)),
\]
which is precisely the gradient flow \eqref{eq:ode-grad}, and
hence equivalent to the ODE system \eqref{eq:ode}.
\end{proof}

\section*{Acknowledgements}

The work of C.C. is partially supported by the National Science Foundation under grant number DMS-2348139.
C.C. would like to thank Academia Sinica for hosting him during his research visit in Dec.~2025.
The work of C.M. is partially supported by the National Science and Technology Council of Taiwan under grant number 113WIA0110762.

\bibliography{refs}

\providecommand{\bysame}{\leavevmode\hbox to3em{\hrulefill}\thinspace}
\providecommand{\MR}{\relax\ifhmode\unskip\space\fi MR }
\providecommand{\MRhref}[2]{%
  \href{http://www.ams.org/mathscinet-getitem?mr=#1}{#2}
}
\providecommand{\href}[2]{#2}
\begin{thebibliography}{10}

\bibitem{AKM1}
S.~Andraus, M.~Katori, and S.~Miyashita, \emph{Interacting particles on the line and {D}unkl intertwining operator of type {A}: Application to the freezing regime}, Journal of Physics A: Mathematical and Theoretical \textbf{45} (2012), 395201.

\bibitem{AKM2}
\bysame, \emph{Two limiting regimes of interacting {B}essel processes}, Journal of Physics A: Mathematical and Theoretical \textbf{47} (2014), 235201.

\bibitem{Andraus-Miyashita}
S.~Andraus and S.~Miyashita, \emph{Two-step asymptotics of scaled {D}unkl processes}, Journal of Mathematical Physics \textbf{56} (2015), 103302.

\bibitem{Andraus-Voit-CLT}
S.~Andraus and M.~Voit, \emph{Central limit theorems for multivariate {B}essel processes in the freezing regime {II}: The covariance matrices}, Journal of Approximation Theory \textbf{246} (2019), 65--84.

\bibitem{AnkerDunklNotes}
J.-P. Anker, \emph{An introduction to {Dunkl} theory and its analytic aspects}, Analytic, Algebraic and Geometric Aspects of Differential Equations (G.~Filipuk, Y.~Haraoka, and S.~Michalik, eds.), Birkh\"auser, Basel, 2015, \url{https://arxiv.org/abs/1611.08213}, pp.~3--58.

\bibitem{arizmendi2018cumulants}
Octavio Arizmendi and Daniel Perales, \emph{Cumulants for finite free convolution}, Journal of Combinatorial Theory, Series A \textbf{155} (2018), 244--266.

\bibitem{BenaychGeorges2007}
Florent Benaych-Georges, \emph{Infinitely divisible distributions for rectangular free convolution: classification and matricial interpretation}, Probability Theory and Related Fields \textbf{139} (2007), 143--189.

\bibitem{benaych2009rectangular}
\bysame, \emph{Rectangular random matrices, related convolution}, Probability Theory and Related Fields \textbf{144} (2009), 471--515.

\bibitem{BenaychGeorges2009entropy}
\bysame, \emph{Rectangular random matrices, related free entropy and free {F}isher's information}, Journal of Operator Theory \textbf{62} (2009), 371--419.

\bibitem{biane1997free}
Philippe Biane, \emph{Free convolution with a semi-circular distribution}, Bulletin des Sciences Math{\'e}matiques \textbf{121} (1997), no.~2, 117--136.

\bibitem{BianeVoiculescu2001}
Philippe Biane and Dan Voiculescu, \emph{A free probability analogue of the {Wasserstein} metric on the trace-state space}, Geometric and Functional Analysis \textbf{11} (2001), no.~6, 1125--1138.

\bibitem{calogero1978motion}
Francesco Calogero, \emph{Motion of poles and zeros of special solutions of nonlinear and linear partial differential equations and related ``solvable'' many-body problems}, Il Nuovo Cimento B (1971-1996) \textbf{43} (1978), no.~2, 177--241.

\bibitem{calogero2001classical}
\bysame, \emph{Classical many-body problems amenable to exact treatments}, Lecture Notes in Physics Monographs, vol.~66, Springer-Verlag, Berlin Heidelberg, 2001.

\bibitem{Chib-SP}
O.~Chybiryakov, \emph{Skew-product representations of multidimensional {D}unkl {M}arkov processes}, Annales de l'Institut Henri Poincar\'e -- Probabilit\'es et Statistiques \textbf{44} (2008), 593--611, \url{https://arxiv.org/abs/0808.3033}.

\bibitem{cole1951}
Julian~D. Cole, \emph{On a quasi-linear parabolic equation occurring in aerodynamics}, Quarterly of Applied Mathematics \textbf{9} (1951), no.~3, 225--236.

\bibitem{csordas1994lehmer}
George Csordas, Wayne Smith, and Richard~S. Varga, \emph{Lehmer pairs of zeros, the {D}e {B}ruijn-{N}ewman constant {$\Lambda$}, and the {R}iemann {H}ypothesis}, Constructive Approximation \textbf{10} (1994), no.~1, 107--129.

\bibitem{Cuenca2024}
Cesar Cuenca, \emph{Cumulants in rectangular finite free probability and beta-deformed singular values}, arXiv preprint (2024), \url{https://arxiv.org/abs/2409.04305}.

\bibitem{Demni-DunklProcs}
N.~Demni, \emph{Radial {D}unkl processes: existence, uniqueness and hitting time}, Comptes Rendus Math\'ematique \textbf{347} (2009), 1125--1128.

\bibitem{Dunkl2}
C.~F. Dunkl, \emph{Differential–difference operators associated to reflection groups}, Transactions of the American Mathematical Society \textbf{311} (1989), 167--183.

\bibitem{edelman2005polynomial}
Alan Edelman and N.~Raj Rao, \emph{The polynomial method for random matrices}, Foundations of Computational Mathematics \textbf{5} (2005), no.~4, 341--391.

\bibitem{GY1}
L.~Gallardo and M.~Yor, \emph{Some new examples of {M}arkov processes which enjoy the time-inversion property}, Probability Theory and Related Fields \textbf{132} (2005), 150--162.

\bibitem{GY2}
\bysame, \emph{A chaotic representation property of the multidimensional {D}unkl processes}, Annals of Probability \textbf{34} (2006), 1530–1549, \url{https://arxiv.org/abs/math/0609679}.

\bibitem{gribinski2024theory}
Aur{\'e}lien Gribinski, \emph{A theory of singular values for finite free probability}, Journal of Theoretical Probability \textbf{37} (2024), no.~2, 1257--1298.

\bibitem{gribinski2022rectangular}
Aur{\'e}lien Gribinski and Adam~W. Marcus, \emph{A rectangular additive convolution for polynomials}, Combinatorial Theory \textbf{2} (2022), no.~1.

\bibitem{guionnet2021large}
A.~Guionnet and J.~Huang, \emph{Asymptotics of rectangular spherical integrals}, Journal of Functional Analysis \textbf{285} (2023), 110144, \url{https://arxiv.org/abs/2106.07146}.

\bibitem{hall2025heat}
Brian~C. Hall and Ching-Wei Ho, \emph{The heat flow conjecture for polynomials and random matrices}, Letters in Mathematical Physics \textbf{115} (2025), no.~3, 60.

\bibitem{hall2025heatgaf}
Brian~C. Hall, Ching-Wei Ho, Jonas Jalowy, and Zakhar Kabluchko, \emph{The heat flow, {GAF}, and {SL}(2; {R})}, Indiana University Mathematics Journal \textbf{74} (2025), 1153--1206.

\bibitem{hall2025zeros}
\bysame, \emph{Zeros of random polynomials undergoing the heat flow}, Electronic Journal of Probability \textbf{30} (2025), 1--55.

\bibitem{hoefert2025zeros}
Antonia H{\"o}fert, Jonas Jalowy, and Zakhar Kabluchko, \emph{Zeros of polynomial powers under the heat flow}, arXiv preprint (2025), \url{https://arxiv.org/abs/2512.17808}.

\bibitem{hopf1950}
Eberhard Hopf, \emph{The partial differential equation $u_t + uu_x = \mu u_{xx}$}, Communications on Pure and Applied Mathematics \textbf{3} (1950), no.~3, 201--230.

\bibitem{HuangInauenMenon2023}
Ching-Peng Huang, Dominik Inauen, and Govind Menon, \emph{Motion by mean curvature and {D}yson {B}rownian motion}, Electronic Communications in Probability \textbf{28} (2023), article no.\ 34, 1--10.

\bibitem{HuangMcSwiggen2024}
Jiaoyang Huang and Colin McSwiggen, \emph{Asymptotics of generalized {B}essel functions and weight multiplicities via large deviations of radial {D}unkl processes}, Probability Theory and Related Fields \textbf{190} (2024), 941--1006.

\bibitem{JordanKinderlehrerOtto1998}
Richard Jordan, David Kinderlehrer, and Felix Otto, \emph{The variational formulation of the {Fokker--Planck} equation}, SIAM Journal on Mathematical Analysis \textbf{29} (1998), no.~1, 1--17.

\bibitem{kabluchko2025lee}
Zakhar Kabluchko, \emph{{L}ee--{Y}ang zeroes of the {C}urie--{W}eiss ferromagnet, unitary {H}ermite polynomials, and the backward heat flow}, Annales Henri Lebesgue \textbf{8} (2025), 1--34.

\bibitem{KornyikMichaletzky}
Mikl{\'o}s Kornyik and Gy{\"o}rgy Michaletzky, \emph{On the moments of roots of {L}aguerre-polynomials and the {M}archenko-{P}astur law}, Annales Universitatis Scientiarum Budapestinensis de Rolando E{\"o}tv{\"o}s Nominatae, Sectio Computatorica \textbf{46} (2017), 137--152.

\bibitem{lasalle1961stability}
J.P. LaSalle and S.~Lefschetz, \emph{Stability by {L}iapunov's direct method: With applications}, Mathematics in science and engineering, Academic Press, New York, 1961.

\bibitem{marcus2021polynomial}
Adam~W. Marcus, \emph{Polynomial convolutions and (finite) free probability}, arXiv preprint (2021), \url{https://arxiv.org/abs/2108.07054}.

\bibitem{MarcusSpielmanSrivastava2022}
Adam~W. Marcus, Daniel~A. Spielman, and Nikhil Srivastava, \emph{Finite free convolutions of polynomials}, Probability Theory and Related Fields \textbf{182} (2022), no.~3, 807--848.

\bibitem{OlshanetskyPerelomov1981}
Mikhail~A. Olshanetsky and Askold~M. Perelomov, \emph{Classical integrable finite-dimensional systems related to {Lie} algebras}, Physics Reports \textbf{71} (1981), 313--400.

\bibitem{Otto2001}
Felix Otto, \emph{The geometry of dissipative evolution equations: the porous medium equation}, Communications in Partial Differential Equations \textbf{26} (2001), no.~1-2, 101--174.

\bibitem{Pac03}
Tommaso Pacini, \emph{Mean curvature flow, orbits, moment maps}, Transactions of the American Mathematical Society \textbf{355} (2003), no.~8, 3343--3357.

\bibitem{RevuzYor1999}
Daniel Revuz and Marc Yor, \emph{Continuous martingales and {B}rownian motion}, 3rd ed., Grundlehren der mathematischen Wissenschaften, vol. 293, Springer, 1999.

\bibitem{Rosler-DunklOps}
M.~R\"osler, \emph{Dunkl operators: Theory and applications}, Orthogonal polynomials and special functions (Leuven, 2002) (E.~Koelink and W.~Van~Assche, eds.), Lecture Notes in Mathematics, vol. 1817, Springer-Verlag, Berlin, 2003, \url{https://arxiv.org/abs/math/0210366}, pp.~93--135.

\bibitem{RoslerVoit-Markov}
M.~R\"osler and M.~Voit, \emph{Markov processes related with {D}unkl operators}, Advances in Applied Mathematics \textbf{21} (1998), 575--643.

\bibitem{sokal2023motion}
Alan~D. Sokal, \emph{Motion of zeros of polynomial solutions of the one-dimensional heat equation: {A} first-order {C}alogero--{M}oser system}, Seminar at University College London (UCL), February 2023, \url{https://mediacentral.ucl.ac.uk/Player/0HiI68e9}. Accessed 29 May 2026.

\bibitem{tao2017heatflow}
Terence Tao, \emph{Heat flow and zeroes of polynomials}, What's new (blog), October 2017, \url{https://terrytao.wordpress.com/2017/10/17/heat-flow-and-zeroes-of-polynomials/}. Accessed 29 May 2026.

\bibitem{voiculescu1991limit}
Dan-Virgil Voiculescu, \emph{Limit laws for random matrices and free products}, Inventiones mathematicae \textbf{104} (1991), no.~1, 201--220.

\bibitem{voiculescu1994free}
\bysame, \emph{The analogues of entropy and of {F}isher's information measure in free probability theory, {II}}, Inventiones mathematicae \textbf{118} (1994), no.~1, 411--440.

\bibitem{voiculescu1992free}
Dan-Virgil Voiculescu, Kenneth~J Dykema, and Alexandru Nica, \emph{Free random variables: A noncommutative probability approach to free products with applications to random matrices, operator algebras and harmonic analysis on free groups}, CRM Monograph Series, vol.~1, American Mathematical Society, 1992.

\bibitem{Voit-CLT}
M.~Voit, \emph{Central limit theorems for multivariate {B}essel processes in the freezing regime}, Journal of Approximation Theory \textbf{239} (2019), 210--231.

\bibitem{VW-limits}
M.~Voit and J.~Woerner, \emph{Limit theorems for {B}essel and {D}unkl processes of large dimensions and free convolutions}, Stochastic Processes and their Applications \textbf{143} (2022), 207--253, \url{https://arxiv.org/abs/2009.13928}.

\bibitem{xu2023rectangular}
Jiaming Xu, \emph{Rectangular matrix additions in low and high temperatures}, Symmetry, Integrability and Geometry: Methods and Applications \textbf{22} (2026), 053, \url{https://arxiv.org/abs/2303.13812}.

\end{thebibliography}
\bibliographystyle{amsplain}

\end{document}